\documentclass[11pt,a4paper]{article}
\usepackage{amsmath,amssymb,amsthm} 
\usepackage{esint,bm,bookmark}
\usepackage{lscape}
\usepackage{hyperref}
\usepackage{mathrsfs}
\usepackage{mathtools}
\usepackage{color}
\usepackage{tikz}
\allowdisplaybreaks[3]

\newtheorem{definition}{Definition}[section]
\newtheorem{theorem}[definition]{Theorem}
\newtheorem{lemma}[definition]{Lemma}
\newtheorem{proposition}[definition]{Proposition}

\theoremstyle{remark}

\numberwithin{equation}{section}

\title{Besov Regularity of Weak solutions to a Class of Nonlinear Elliptic Equations}

\author{Huimin Cheng \\
School of Mathematics and Statistics \\
Shandong Normal University\\
Jinan, Shandong, 250358, China\\
e-mail:\,\texttt{2995246312@qq.com}\\
and\\
Feng Zhou\footnote{F.~Zhou is the corresponding author.} \\
School of Mathematics and Statistics \\
Shandong Normal University\\
Jinan, Shandong, 250358, China\\
e-mail:\,\texttt{zhoufeng@u.nus.edu} } 
\date{\today}

\begin{document}
\maketitle
\begin{abstract}
In this article, we study a Besov regularity estimate of weak solutions to a class of nonlinear elliptic equations in divergence form. The main purpose is to establish Calder\'on-Zygmund type estimate in Besov spaces with more general assumptions on coefficients, non-homogeneous term and integrable index. By involving the Sharp maximal function, we establish an oscillation estimate of weak solutions in Orlicz-Sobolev spaces. By deriving a higher integrability estimate of weak solutions, we obtain the desired regularity estimate which expands the Calder\'on-Zygmund theory for nonlinear elliptic equations. \\

Keywords: Nonlinear elliptic equations; Higher integrability; Regularity; Besov spaces.\\

Mathematics Subject classification: 35B65, 35J60.
\end{abstract}

\newpage

\section{Introduction}
The main purpose of this article is to obtain the Besov regularity estimate of weak solutions to a class of nonlinear elliptic equations in divergence form 
\begin{equation}\label{equ}
\operatorname{div} A\left(x,\nabla u\right)=\operatorname{div}\left[\,\frac{P(|F|)}{|F|^2}F\,\right]
\end{equation}
in $\Omega$, where $\Omega\subset\mathbb{R}^n$ ($n\geq2$) is an open domain which is bounded, the unknown $u\in W^{1,P}_{{\rm{loc}}}(\Omega): \Omega\rightarrow\mathbb{R}$, and $P(t)\,(t\geq0)$ is an $N$-function. The Orlicz-Sobolev spaces $W^{1,P}(\Omega)$ and the assumption of $N$-function will be introduced in Section 2. 
In \eqref{equ},  $F:\Omega\rightarrow\mathbb{R}^n$, and $A:\Omega\times\mathbb{R}^n\rightarrow\mathbb{R}^n$ is a Carath\'eodory vector field with general growth and uniformly elliptic conditions, that is, there exist constants $\upsilon,\,L,\,k>0$ and $0<\mu<1$ such that
\begin{equation}
\begin{split}
\begin{aligned}
&\textrm{(A1)}\qquad \left[A(x,\xi)-A(x,\eta)\right]\cdot(\xi-\eta)\geq\upsilon\frac{P\left[\left(\mu^2+|\xi|^2+|\eta|^2\right)^{\frac{1}{2}}\right]}{\mu^2+|\xi|^2+|\eta|^2}|\xi-\eta|^2\,,\nonumber\\
&\textrm{(A2)}\qquad \left|A(x,\xi)-A(x,\eta)\right|\leq L\frac{P\left[\left(\mu^2+|\xi|^{2}+|\eta|^2\right)^{\frac{1}{2}}\right]}{\mu^2+|\xi|^{2}+|\eta|^{2}}|\xi-\eta|\,,\nonumber\\
&\textrm{(A3)}\qquad |A(x,\xi)|\leq k\frac{P\left[\left(\mu^2+|\xi|^2\right)^{\frac{1}{2}}\right]}{(\mu^2+|\xi|^2)^{\frac{1}{2}}}\nonumber\\
\end{aligned}
\end{split}
\end{equation}
for every $\xi,\,\eta\,\in\mathbb{R}^n$ and for almost all $x\in\,\Omega$. \\
\indent The regularity of solutions to elliptic equations was traced back to the Weyl's lemma. It was the first result in the development of regularity theories of elliptic equations, and it was extended to the case of general elliptic problems. Later on in relevant papers (\cite{9,10,12-3,12-4,12-5,12-6}),  the $W^{1,p}$ estimates for elliptic equations were proved. With the deepening of theoretical research, the H\"{o}lder regularity theories of elliptic equations in divergence form were widely studied. At present, the studies of elliptic equations are concerned with the regularity estimates of weak solutions in Besov spaces (\cite{18-1,18-2,18-3}). Compared with the classical Sobolev spaces, Besov spaces consist of a wide class of functions. \\
\indent Bais\'on A.~L. (\cite{1}) studied a class of nonlinear elliptic equations in divergence form
\begin{equation}
\operatorname{div}A(x,Du)=G, \nonumber
\end{equation}
and obtained a Calder\'on-Zygmund type regularity of weak solutions in Besov spaces. Then Clop A. (\cite{2}) extended the result in Besov spaces. Moreover, In \cite{3}, Lyaghfouri A. studied the equation
\begin{equation}\label{AL-eqn}
\operatorname{div}\left[\frac{a(|\nabla u|)}{|\nabla u|}\nabla u\right]=\operatorname{div}\left[\frac{a(|F|)}{|F|}F\right],\ \ u\in W^{1,P}(\mathbb{R}^n).
\end{equation}
It should be mentioned that the author established a higher integrability of weak solutions to \eqref{AL-eqn}. \\
\indent In this article, we assume that there exists a sequence of measurable non-negative functions $g_k\in L^{\frac{n}{\alpha}}(\Omega)\,(k\in\mathbb{N},\ 0<\alpha<1)$ satisfying that
\begin{equation} 
\textrm{(A4)}\quad 
\left\{
\begin{aligned}
&\ \sum_k\|g_k\|^q_{L^{\frac{n}{\alpha}}(\Omega)}<\infty,\,(1\leq q<\infty)\nonumber\\ 
&\ \left|A(x,\xi)-A(y,\xi)\right|\leq|x-y|^{\alpha}\left(g_{k}(x)+g_{k}(y)\right)\frac{P\left[\left(\mu^2+|\xi|^2\right)^{\frac{1}{2}}\right]}{\left(\mu^2+|\xi|^2\right)^{\frac{1}{2}}} \nonumber
\end{aligned}
\right.
\end{equation}
for each $\xi\in\mathbb{R}^n$, and almost every $x, y\in\Omega$ such that $2^{-k}\leq|x-y|<2^{-k+1}$. According to (A4), we write $(g_k)_k\in l^q(L^{\frac{n}{\alpha}}(\Omega))$ in short. \\
\indent In particular, we assume that there exists a function $g\in L^{\frac{n}{\alpha}}(\Omega)\,(0<\alpha<1)$ such that
\begin{equation}\label{1.4}
|A(x,\xi)-A(y,\xi)|\leq|x-y|^\alpha(g(x)+g(y))\frac{P\left[\left(\mu^2+|\xi|^2\right)^{\frac{1}{2}}\right]}{\left(\mu^2+|\xi|^2\right)^{\frac{1}{2}}}
\end{equation}
for almost every $x, y\in\Omega$ and all $\xi\in\mathbb{R}^n$. By introducing an auxiliary function
\begin{equation}\label{Vp-def}
V_P(\xi)=\left[\frac{P\left[\left(\mu^2+|\xi|^2\right)^{\frac{1}{2}}\right]}{\mu^2+|\xi|^2}\right]^{\frac{1}{2}}\xi
\end{equation}
with $\xi\in\mathbb{R}^{n}$, we present the main results of this article. 
\begin{theorem}\label{theorem1.3}
Let $0\!<\alpha<\!1$. Assume that $A$ satisfies hypotheses {\rm{(A1)}}-{\rm{(A3)}} and \eqref{1.4} with $0<\mu<1$, and $P(t)$ satisfies {\rm{(P1)}} and {\rm{(P2)}}. If $u\in W^{1,P}_{\rm{loc}}(\Omega)$ is a weak solution to
\begin{equation}\label{1.5}
{\rm{div}}A(x,\nabla u)=0\qquad{\rm{in}}\,\,\Omega,
\end{equation}
then $V_P(\nabla u)\in B^\alpha_{2,\infty}(\Omega)$ locally.
\end{theorem}
\begin{theorem}\label{MainTheorem}
Let $0\!<\!\alpha<\!1$, and $1\leq q<\frac{2n}{n-2\alpha}$. Assume that the hypotheses {\rm{(A1)}}-{\rm{(A4)}} and {\rm{(P1)}}-{\rm{(P2)}} hold. If $u\in W^{1,P}_{\rm{loc}}(\Omega)$ is a weak solution to \eqref{equ} with $0<\mu<1$ and $\frac{P(|F|)}{|F|^2}F\in B^\alpha_{2,q}(\Omega)$, then $V_P(\nabla u)\in B^\alpha_{2,q}(\Omega)$ locally. 
\end{theorem}
See Section 2 for the definitions of $W^{1,P}(\Omega)$ and $B^{\alpha}_{2,q}(\Omega)$, and the assumptions {\rm{(P1)}}-{\rm{(P2)}} of $N$-functions. \\
\indent The contribution of the main results is to study a wide class of elliptic equations in divergence form which is a generalization of classical forms. Our aim is to obtain a Besov regularity estimate of weak solutions, hence the well-known estimates could become a corollary. The hypotheses (A1)-(A4) shall be an extension of the VMO conditions. Meanwhile, we study the weak solutions in the $W^{1,P}(\Omega)$ spaces with a variable index $P(t)$, which could cover the case for a real value $p$. \\
\indent This article is organized as follows. In section 2 we give some technical lemmas and tools such as classical inequalities. In section 3, we establish Propositions relating to the higher integrability of weak solutions. In section 4, we present the proofs of Theorem \ref{theorem1.3} and Theorem \ref{MainTheorem}, respectively.

\section{Preliminary}

\subsection{$N$-function}
\begin{definition}[\cite{22}]
We say that a function $P(t)$: $[0,\infty)\rightarrow[0,\infty)$ is an $N$-function, if it satisfies the following conditions:\\ 
(1) There exist derivatives $P'(t)$ and $P''(t)$ of $P(t)$ such that $P(0)=0$, $P(1)=1$, $P'(t)>0$ and $P''(t)>0$ for $t>0$. \\
(2) ($\Delta_2$-condition) There exists $c_1>0$ such that $P(2t)\leq c_1P(t)$ for all $t\geq0$. By $\Delta_2(P)$ we denote the smallest constant $c_1$.
\end{definition}
We write $f\sim g$ if and only if there exist constants $C_{1},C_{2}>0$ such that $C_{1}f\leq g\leq C_{2}f$. Let $P(t)$ be an $N$-function. It is obvious that $P(2t)\sim P(t)$. Defining the complementary function $\tilde{P}$ of $P$ by 
\begin{equation}\label{2.1}
\tilde{P}(s)=\sup_{t\geq0}\{st-P(t)\},
\end{equation}
one finds that 
\begin{equation}\label{2.3}
P(t)\sim P'(t)t\ \ \text{and}\ \ \tilde{P}\left(\frac{P(t)}{t}\right)\sim P(t),
\end{equation}
uniformly in $t\geq0$, see \cite{22}. \\
\indent For every $\epsilon>0$, there exists $C(\epsilon,\Delta_2(\tilde{P}))$ such that for all $t, u\geq0$, there holds
\begin{equation}\label{young1}
t\,u\leq\epsilon\,P(t)+C(\epsilon,\Delta_2(\tilde{P}))\tilde{P}(u),
\end{equation}
which is a Young type inequality (\cite{22}). In particular, 
\begin{equation}\label{young2}
ab\leq\epsilon\,a^2+C(\epsilon)\,b^2,
\end{equation}
with $a, b>0$. \\
\indent Throughout this article, we assume that $N$-function $P(t)$ satisfies the following two assumptions:\\
(P1) $P(s t)\sim P(s)P(t)$, $\tilde{P}(st)\sim \tilde{P}(s)\tilde{P}(t)$, $\frac{P(t)}{t^2}$ is monotone increasing on $(0,+\infty)$. \\
(P2) There is a positive constant $\hat{C}$ such that
\begin{equation}
\hat{C}^{-1}\,\frac{P\left[(\mu^2+|\xi|^2+|\eta|^2)^{\frac{1}{2}}\right]}{\mu^2+|\xi|^2+|\eta|^2}\leq\frac{|V_P(\xi)-V_P(\eta)|^2}{|\xi-\eta|^2}\leq \hat{C}\,\frac{P\left[(\mu^2+|\xi|^2+|\eta|^2)^{\frac{1}{2}}\right]}{\mu^2+|\xi|^2+|\eta|^2}
\end{equation}
for any $\xi,\,\eta\in\mathbb{R}^n$ and $|\xi-\eta|\neq 0$.\\
\indent At the end of this subsection, by involving the $N$-function, we present a modified version of the H\"older inequality as 
\begin{eqnarray}\label{2.7}
\int_{\Omega}|f(x)g(x)|\operatorname{d}\!x&\leq&C\,\tilde{P}^{-1}\left(\int_{\Omega}\tilde{P}(|f(x)|)\operatorname{d}\!x\right)\cdot P^{-1}\left(\int_{\Omega}P(|g(x)|)\operatorname{d}\!x\right).
\end{eqnarray}
Here $P^{-1}(P(t))=t$ and $\tilde{P}(t)$ is defined by \eqref{2.1}. 

\subsection{Function Spaces}
By using an $N$-function $P(t)$, we recall the definition of Orlicz spaces $L^P(\mathbb{R}^n)$ with its norm (\cite{5}):
\begin{eqnarray*}
L^P(\mathbb{R}^n)&=&\left\{u(x)\in L^1(\mathbb{R}^n)\,\bigg|\,\int_{\mathbb{R}^n}P(|u(x)|)\operatorname{d}\!x<\infty\right\},\\
\|u\|_{L^{P}(\mathbb{R}^n)}&=&\inf\left\{k>0\,\bigg|\,\int_{\mathbb{R}^n}P\left(\frac{|u(x)|}{k}\right)\operatorname{d}\!x<1\right\}.
\end{eqnarray*}
The dual space of $L^P(\mathbb{R}^n)$ is $L^{\tilde{P}}(\mathbb{R}^n)$, where $\tilde{P}$ is introduced in \eqref{2.1}. \\
\indent The Orlicz-Sobolev spaces $W^{1,P}(\mathbb{R}^n)$ with its norm are given by (\cite{6})
\begin{eqnarray*}\label{W1P}
W^{1,P}(\mathbb{R}^n)&=&\left\{u\in L^P(\mathbb{R}^n)\,\big|\,|\nabla u|\in L^P(\mathbb{R}^n)\right\},\\
\|u\|_{W^{1,P}(\mathbb{R}^n)}&=&\|u\|_{L^{P}(\mathbb{R}^n)}+\|\nabla u\|_{L^{P}(\mathbb{R}^n)}\,.\nonumber
\end{eqnarray*}
It is clear that a function $u\in W^{1,P}_{\text{loc}}(\Omega)$, if $u\in W^{1,P}(\Omega_0)$ for every $\Omega_0\Subset\Omega$ (\cite{21}). \\
\indent The Besov spaces $B^{\alpha}_{p, q}(\mathbb{R}^n)$ with its norm are defined via (\cite{7})
\begin{eqnarray*}
\|u\|_{B^\alpha_{p, q}(\mathbb{R}^n)} &=&\|u\|_{L^p(\mathbb{R}^n)}+\left[u\right]_{B^{\alpha}_{p, q}(\mathbb{R}^n)}<\infty,\quad \left(0<\alpha<1,\ 1\leq p<\infty\right) \nonumber\\
\left[u\right]_{B^{\alpha}_{p, q}(\mathbb{R}^n)} &=&
\left\{ 
\begin{array}{ll}
\displaystyle \left(\int_{\mathbb{R}^n}\left(\int_{\mathbb{R}^n}\frac{|\Delta_{h} u|^p}{|h|^{\alpha p}}\operatorname{d}\!x\right)^{\frac{q}{p}}\frac{\operatorname{d}\!h}{|h|^n}\right)^{\frac{1}{q}}<\infty, & 1\leq q<\infty; \\[0.1cm]
\displaystyle \sup_{h\in\mathbb{R}^n}\left(\int_{\mathbb{R}^n}\frac{|\Delta_{h} u|^p}{|h|^{\alpha p}}\operatorname{d}\!x\right)^{\frac{1}{p}}<\infty, & q=\infty. 
\end{array}
\right. 
\end{eqnarray*}
Here $\Delta_{h}u=u(x+h)-u(x)$.

\subsection{Hardy-Littlewood Maximal Function and Sharp Maximal Function}
\begin{definition}\label{hardy}
Let $f\in L^1_{\operatorname{loc}}(\mathbb{R}^n)$, we define the Hardy-Littlewood maximal function of $f$ by
\begin{equation}
\mathcal{M}f(x)=\sup_{r>0}\frac{1}{|B_{r}(x)|}\int_{B_{r}(x)}|f(y)|\operatorname{d}\!y=\sup_{r>0}\fint_{B_{r}(x)}|f(y)|\operatorname{d}\!y. \nonumber
\end{equation}
Here $\fint_{B}=\frac{1}{|B|}\int_{B}$ denotes the mean of integral. By using the $N$-function $P(t)$, we define 
\begin{equation}
\mathcal{M}_Pf(x)=\sup_{r>0}P^{-1}\left(\fint_{B_{r}(x)}P(|f(y)|)\operatorname{d}\!y\right).\nonumber
\end{equation}
\end{definition}

\begin{definition}[\cite{24}]\label{sharp}
Assume $f\in L^1_{\operatorname{loc}}(\mathbb{R}^n)$, we define 
\begin{equation}
\mathcal{M}^\#f(x)=\sup_{r>0}\fint_{B_{r}(x)}|f(y)-f_{B_{r}(x)}|\operatorname{d}\!y\nonumber
\end{equation}
as the Sharp maximal function of $f$, where $f_{B_{r}(x)}$ is the mean value of $f$ in $B_{r}(x)$.
\end{definition}
One can show that, if $f\in L^P(\mathbb{R}^n)$, then $\mathcal{M}f\in L^P(\mathbb{R}^n)$, and there exist positive constants $C_{1},C_{2}>0$ and $\vartheta\geq2$ with $B_{\vartheta R}\subset\mathbb{R}^n$ such that 
\begin{eqnarray}\label{M-Msharp}
\int_{B_R}P\left(\mathcal{M}f(x)\right)\operatorname{d}\!x\leq C_{1}\int_{B_R}P\left(|f(x)|\right)\operatorname{d}\!x\leq C_{2}\int_{B_{\vartheta R}}P\left(\mathcal{M}^{\#}f(x)\right)\operatorname{d}\!x.
\end{eqnarray}

\subsection{Weak Solution}
We present the definition of the weak solutions to \eqref{equ}. 
\begin{definition}\label{weak de}
If for any $\varphi\in C^\infty_0(\Omega)$, there holds
\begin{equation}\label{weak}
\int_{\Omega} A(x, \nabla u)\cdot\,\nabla\varphi\operatorname{d}\!x=\int_{\Omega}\frac{P(|F|)}{|F|^2}F\,\cdot\,\nabla\varphi\operatorname{d}\!x, 
\end{equation}
then $u\in W^{1,P}_{{\rm{loc}}}(\Omega)$ is a weak solution to \eqref{equ}. Here we call $\varphi$ is a test function. 
\end{definition}

\subsection{Iteration Lemma}
The following lemma will play an important role in the proof of main theorem.  
\begin{lemma}[\cite{GM}]\label{hole}
Let $a: [r, R_0]\rightarrow\mathbb{R}$ be a non-negative bounded function, and $0<\nu<1$, $A,\,B\geq1$, $\beta>0$. Assume that
\begin{eqnarray*}
a(s)\leq\nu\,a(t)+\frac{A}{(t-s)^\beta}+B.
\end{eqnarray*}
for all $r\leq s<t\leq R_0$. Then
\begin{eqnarray*}
a(r)\leq\frac{c\,A}{(R_0-r)^\beta}+c\,B,
\end{eqnarray*}
where $c=c(\nu, \beta)>0$.
\end{lemma}

\section{Oscillation Estimate and Higher Integrability}
In this section, whenever there is no risk of misunderstanding, we suppress $\operatorname{d}\!x$ from the integral express on domains. We note that the practice on using $C$, possible with subindices, to denote various positive constants, which may be rendered differently from line to line according to the contents. 
\subsection{Oscillation Estimate}
For the vector field $A(x, \xi)$ appeared in \eqref{equ}, we let 
\begin{equation}\label{equ-ABR}
A_{B}(\xi)=\fint_{B}A(x,\xi)\operatorname{d}\!x 
\end{equation}
for $\xi\in\mathbb{R}^{n}$ and a ball $B\subset\Omega$. Then we define  
\begin{equation}\label{Vx-BR}
V(x,B)=\sup_{\xi\in\mathbb{R}^{n}}\frac{|A(x,\xi)-A_{B}(\xi)|}{\frac{P\left[\left(\mu^2+|\xi|^2\right)^{\frac{1}{2}}\right]}{\left(\mu^2+|\xi|^2\right)^{\frac{1}{2}}}}, 
\end{equation}
where $B\subset\Omega$ is a ball and $x\in\Omega$. Inspired by \cite{1}, we adopt an oscillation approach and use the sharp maximal function to get the following. 
\begin{proposition}\label{Osc-estimate}
Let Q(t) be an $N$-function such that $Q\circ P^{-1}$ is an $N$-function. Assume that $A$ satisfies {\rm{(A1), (A2), (A3)}} and is locally uniformly in VMO, and the $N$-function $P(t)$ satisfies {\rm{(P1)}}. There exists $\lambda > 1$ with the following property. If $x_0 \in \Omega$, then there is a number $d_0 >0$ such that if $u \in W^{1,Q}(\Omega)$ is a weak solution to \eqref{equ} for some $F\in L^Q_{\operatorname{loc}}(\Omega)$, then the estimate
\begin{equation}\label{lemma2}
\fint_{\frac{1}{2}B_d}Q(|\nabla u|)\leq C\left[Q(\mu)+\fint_{\lambda B_{d}}Q(|F|)+\frac{1}{Q(d)}\fint_{\lambda B_{d}}Q(|u|)\right]
\end{equation}
holds with $0 < d < d_0$, $B_{d}=B_{d}(x_0)$, and $\lambda B_d\subset\Omega$. 
\end{proposition}

\begin{proof}
We are given a ball $B_d=B_{d}(x_0)$ satisfying that $\lambda B_{d}:=\left(1+\frac{2}{\delta}\right)k_0B_d\subset\Omega$, where $k_0\geq2$ and $\delta\in(0, 1)$ which is determined later. \\
\textbf{Step 1.} We first localize the problem. We take arbitrary radii $0 < \frac{d}{2}<\rho<r<d$ with balls $B_{\rho}=B_{\rho}(x_0)$ and $B_{r}=B_{r}(x_0)$. There exists a cut-off function $\eta: \mathbb{R}^n\rightarrow\mathbb{R}$ such that $\eta\in C^{\infty}_0(\mathbb{R}^n)$, $\chi_{B_{\rho}}\leq\eta\leq\chi_{B_{r}}$ and $\|\nabla\eta\|_{L^{\infty}}\leq\frac{c}{r-\rho}$ with $c>0$. By setting $\omega=\tilde{P}(\eta)\,u$, it is clear that $\omega\in W^{1, Q}_{0}(B_{r})$. It follows that
\begin{eqnarray}\label{equ1}
\operatorname{div}A(x,\nabla\omega)\!&\!=\!&\!\operatorname{div}\left[A(x,\nabla\omega)-A(x,\tilde{P}(\eta)\nabla u)\right]\nonumber\\
\!&\!\!&\!+\,\operatorname{div}\left[A(x,\tilde{P}(\eta)\nabla u)-P(\eta)A(x,\nabla u)\right]+\operatorname{div}\big[P(\eta)A(x,\nabla u)\big]\nonumber\\
\!&\!=\!&\!\operatorname{div}\left[A(x,\nabla\omega)-A(x,\tilde{P}(\eta)\nabla u)\right]+\operatorname{div}\left[A(x,\tilde{P}(\eta)\nabla u)-P(\eta)A(x,\nabla u)\right]\nonumber\\[0.1cm]
\!&\!\!&\!+\,P(\eta)\operatorname{div}A(x,\nabla u)+P'(\eta)\nabla\eta\cdot A(x,\nabla u).
\end{eqnarray}
If $y\in k_0B_d$ and $0<R<\frac{k_0d}{\delta}$ with $B_R=B_{R}(y)$, then $B_R\subset\left(1+\frac{1}{\delta}\right)\,k_0B_d\subset\Omega$. Taking $B=B_{R}$ in \eqref{equ-ABR}, one can show that $A_{B_R}(\cdot)$ satisfies hypotheses (A1), (A2) and (A3). Let $\upsilon$ be a weak solution to the following Dirichlet problem
\begin{equation}\label{equ2}
\left\{
\begin{array}{cl}
\operatorname{div}A_{B_R}(\nabla\upsilon)=0 &\text{in}\ B_R\,,\\
\upsilon=\omega &\text{on}\ \partial B_R\,. 
\end{array}
\right.
\end{equation}
By multiplying both sides of \eqref{equ1} by $\upsilon-\omega$, and applying the integration by parts, we obtain
\begin{eqnarray*}
&& \int_{B_R}A(x,\nabla\omega)\cdot(\nabla\upsilon-\nabla\omega)\\
&=&\int_{B_R}\left[A(x,\nabla\omega)-A(x,\tilde{P}(\eta)\nabla u)\right]\cdot(\nabla\upsilon-\nabla\omega)\\
&& +\ \int_{B_R}\left[A(x,\tilde{P}(\eta)\nabla u)-P(\eta)A(x,\nabla u)\right]\cdot(\nabla\upsilon-\nabla\omega)\\
&& +\ \int_{B_R}A(x,\nabla u)\cdot\nabla\big[P(\eta)(\upsilon-\omega)\big]-\int_{B_R}P'(\eta)\nabla\eta\cdot A(x,\nabla u)(\upsilon-\omega). \nonumber 
\end{eqnarray*}
Since $u$ is a weak solution to \eqref{equ}, we get
\begin{eqnarray}\label{equ3}
&&\int_{B_R}A(x,\nabla\omega)\cdot(\nabla\upsilon-\nabla\omega)\nonumber\\
&=& \int_{B_R}\left[A(x,\nabla\omega)-A(x,\tilde{P}(\eta)\nabla u)\right]\cdot(\nabla\upsilon-\nabla\omega)\nonumber\\
&&+\,\,\int_{B_R}\left[A(x,\tilde{P}(\eta)\nabla u)-P(\eta)A(x,\nabla u)\right]\cdot(\nabla\upsilon-\nabla\omega)\nonumber \\
&&+\,\,\int_{B_R}\frac{P(|F|)}{|F|^2}F\cdot\nabla\big[P(\eta)(\upsilon-\omega)\big]-\int_{B_R} P'(\eta)\nabla\eta\cdot A(x,\nabla u)(\upsilon-\omega). 
\end{eqnarray}
As $\upsilon$ solves the Dirichlet problem \eqref{equ2}, one obtains that 
\begin{equation}\label{equ15}
\int_{B_{R}}\big[A_{B_R}(\nabla\upsilon)-A_{B_R}(\nabla\omega)\big]\cdot(\nabla\upsilon-\nabla\omega)=\int_{B_R}A_{B_R}(\nabla\omega)\cdot(\nabla\omega-\nabla\upsilon). 
\end{equation}
Combining \eqref{equ3} and \eqref{equ15}, it follows that 
\begin{eqnarray*}
\!&\!\!&\! \int_{B_R}\big[A_{B_R}(\nabla\upsilon)-A_{B_R}(\nabla\omega)\big]\cdot(\nabla\upsilon-\nabla\omega)\\
\!&\!=\!&\!\int_{B_R}\big[A_{B_R}(\nabla\omega)-A(x,\nabla\omega)\big]\cdot(\nabla\omega-\nabla\upsilon) \\
\!&\!\!&\!+\,\int_{B_R}\left[A(x,\nabla\omega)-A(x,\tilde{P}(\eta)\nabla u)\right]\cdot(\nabla\omega-\nabla\upsilon)\\
\!&\!\!&\!+\,\int_{B_R}\left[A(x,\tilde{P}(\eta)\nabla u)\!-\!P(\eta)A(x,\nabla u)\right]\cdot(\nabla\omega\!-\!\nabla\upsilon)\\
\!&\!\!&\!+\int_{B_R}\frac{P(|F|)}{|F|^2}F\cdot\nabla\big[P(\eta)(\omega-\upsilon)\big]-\int_{B_R}P'(\eta)\nabla\eta\cdot A(x,\nabla u)(\omega-\upsilon). \nonumber
\end{eqnarray*}
By introducing $K_{0}$ and taking the absolute value, we have
\begin{eqnarray}\label{equ16}
K_{0}\!&\!:=\!&\! \int_{B_R}\left[A_{B_{R}}(\nabla\upsilon)-A_{B_{R}}(\nabla\omega)\right]\cdot(\nabla\upsilon-\nabla\omega)\nonumber \\
\!&\!\leq\!&\! \int_{B_R}|A_{B_{R}}(\nabla\omega)-A(x,\nabla\omega)|\,|\nabla\upsilon-\nabla\omega|\nonumber\\
\!&\!\!&\!+\,\int_{B_R}\left|A(x,\nabla\omega)-A(x,\tilde{P}(\eta)\nabla u)\right||\nabla\upsilon-\nabla\omega|\nonumber\\
\!&\!\!&\!+\,\int_{B_R}\left|A(x,\tilde{P}(\eta)\nabla u)\!-\!P(\eta)A(x,\nabla u)\right||\nabla\upsilon-\nabla\omega|\nonumber\\
\!&\!\!&\!+\,\int_{B_R}\frac{P(|F|)}{|F|}\left|\nabla\big[P(\eta)(\upsilon-\omega)\big]\right|+\int_{B_R}|P'(\eta)|\,|\nabla\eta|\,|A(x,\nabla u)|\,|(\upsilon-\omega)|\nonumber \\[0.1cm]
\!&\!=\!&\! K_1+K_2+K_3+K_4+K_5.
\end{eqnarray}
\textbf{Step 2.} We shall estimate each term $K_{i}$ ($0\leq i\leq 5$) displayed in \eqref{equ16}. By (A1) and (P1), we obtain that
\begin{equation*}
K_0\,\geq\,\upsilon\int_{B_R}\frac{P\left[\left(\mu^2+|\nabla\upsilon|^2+|\nabla\omega|^2\right)^{\frac{1}{2}}\right]}{\mu^2+|\nabla\upsilon|^2+|\nabla\omega|^2}\,|\nabla\upsilon-\nabla\omega|^2\,\geq\,\upsilon\int_{B_R}P\left(|\nabla\upsilon-\nabla\omega|\right).
\end{equation*}
By \eqref{young1} and the definition \eqref{Vx-BR} with $B=B_{R}$, we estimate $K_1$ as
\begin{eqnarray}
K_1&\leq&\int_{B_R}\left|V(x,B_R)\,\frac{P\left[\left(\mu^2+|\nabla\omega|^2\right)^{\frac{1}{2}}\right]}{\left(\mu^2+|\nabla\omega|^2\right)^{\frac{1}{2}}}\right|\,|\nabla\upsilon-\nabla\omega|\nonumber \\
&\leq&\epsilon\int_{B_R}P(|\nabla\upsilon-\nabla\omega|)\,+\,C\int_{B_R}\tilde{P}\left[V(x,B_R)\,\frac{P\left[\left(\mu^2+|\nabla\omega|^2\right)^{\frac{1}{2}}\right]}{\left(\mu^2+|\nabla\omega|^2\right)^{\frac{1}{2}}}\right]. \nonumber
\end{eqnarray}
Let $T(t)$ be an $N$-function such that $(T\circ P^{-1})(t)$ is also an $N$-function. By \eqref{2.3} and \eqref{2.7}, we obtain 
\begin{eqnarray*}
K_1&\leq&\epsilon\int_{B_R}P(|\nabla\upsilon-\nabla\omega|)+C\int_{B_R}\tilde{P}\big[V(x,B_R)\big]P\left[\left(\mu^2+|\nabla\omega|^2\right)^{\frac{1}{2}}\right]\\
&\leq&\epsilon\int_{B_R}P(|\nabla\upsilon-\nabla\omega|)+ C\left(\widetilde{T\circ P^{-1}}\right)^{-1}\left[\int_{B_R}\left(\widetilde{T\circ P^{-1}}\right)\left[\tilde{P}\big[V(x,B_R)\big]\right]\right]\\
&& \qquad\qquad\qquad\qquad\qquad \cdot\left(T\circ P^{-1}\right)^{-1}\left[\int_{B_R}\left(T\circ P^{-1}\right)P\left[\left(\mu^2+|\nabla\omega|^2\right)^{\frac{1}{2}}\right]\right]. 
\end{eqnarray*}
Then there exists an $N$-function $G(t)$ such that $(\widetilde{T\circ P^{-1}}\circ\tilde{P})(t)=G(t)\cdot t$. It follows that 
\begin{eqnarray*}
K_1\!&\!\leq\!&\!\epsilon\int_{B_R}P(|\nabla\upsilon-\nabla\omega|)\\
\!&\!\!&\!+\,C\left(\widetilde{T\circ P^{-1}}\right)^{-1}\!\left[\int_{B_R}G\big[V(x,B_R)\big]V(x,B_R)\right]\cdot\left(P\circ T^{-1}\right)\left[\int_{B_R}T\left[\left(\mu^2+|\nabla\omega|^2\right)^{\frac{1}{2}}\right]\right]\\
\!&\!\leq\!&\!\epsilon\int_{B_R}P(|\nabla\upsilon-\nabla\omega|)\\
\!&\!\!&\!+\,C\left(\widetilde{T\circ P^{-1}}\right)^{-1}\!\left[\int_{B_R}V(x,B_R)\right]\cdot\left(P\circ T^{-1}\right)\left[\int_{B_R}T\left[\left(\mu^2+|\nabla\omega|^2\right)^{\frac{1}{2}}\right]\right], \nonumber
\end{eqnarray*}
where $\epsilon>0$ will be chosen later. Here we use the fact that $V(x,B_R)$ is bounded in $\Omega$. For $K_{2}$, according to the definition of $\omega$ and (A2), we acquire 
\begin{eqnarray*}
K_2&\leq&L\int_{B_R}\frac{P\left[\left(\mu^2+|\nabla\omega|^2+\big|\tilde{P}(\eta)\nabla u\big|^2\right)^{\frac{1}{2}}\right]}{\mu^2+|\nabla\omega|^2+\big|\tilde{P}(\eta)\nabla u\big|^2}\,|\nabla\upsilon-\nabla\omega|\,\left|\nabla\omega-\tilde{P}(\eta)\nabla u\right|\\
&=&L\int_{B_R}\frac{P\left[\left(\mu^2+|\nabla\omega|^2+\big|\nabla\omega-\tilde{P}'(\eta)\nabla\eta\,u\big|^2\right)^{\frac{1}{2}}\right]}{\mu^2+|\nabla\omega|^2+\big|\nabla\omega-\tilde{P}'(\eta)\nabla\eta\,u\big|^2}\,|\nabla\upsilon-\nabla\omega|\,\left|\tilde{P}'(\eta)\nabla\eta\,u\right|. \nonumber
\end{eqnarray*}
Using the properties of $N$-function, we have
\begin{eqnarray*}
K_2&\!\leq\!&C\int_{B_R}\frac{P\left[\left(\mu^2+|\nabla\omega|^2+\big|\tilde{P}'(\eta)\nabla\eta\,u\big|^2\right)^{\frac{1}{2}}\right]}{\mu^2+|\nabla\omega|^2+\big|\tilde{P}'(\eta)\nabla\eta\,u\big|^2}\,|\nabla\upsilon-\nabla\omega|\,\left|\tilde{P}'(\eta)\nabla\eta\,u\right|\\
&\!\leq\!&C\int_{B_R}\frac{P\left[\left(\mu^2+|\nabla\omega|^2\right)^{\frac{1}{2}}\right]}{\mu^2+|\nabla\omega|^2}\,|\nabla\upsilon-\nabla\omega|\,\left|\tilde{P}'(\eta)\nabla\eta\,u\right|\\
&&+\,\,C\int_{B_R}\frac{P\left[\,\left|\tilde{P}'(\eta)\nabla\eta\,u\right|\,\right]}{\big|\tilde{P}'(\eta)\nabla\eta\,u\big|^2}\,|\nabla\upsilon-\nabla\omega|\,\left|\tilde{P}'(\eta)\nabla\eta\,u\right|.
\end{eqnarray*}
Applying \eqref{young1} twice and using the assumption on $\nabla\eta$, one deduces that
\begin{eqnarray*}
K_2 &\leq& C\int_{B_R}\frac{P\left[\left(\mu^2+|\nabla\omega|^2\right)^{\frac{1}{2}}\right]}{\mu^2+|\nabla\omega|^2}|\nabla\upsilon-\nabla\omega|\left|\tilde{P}'(\eta)\nabla\eta\,u\right|\\
&&+\,\,C\int_{B_R}\frac{P\left[\,\left|\tilde{P}'(\eta)\nabla\eta\,u\right|\,\right]}{\big|\tilde{P}'(\eta)\nabla\eta\,u\big|}|\nabla\upsilon-\nabla\omega|\\
&\leq&\epsilon\int_{B_R}P(|\nabla\upsilon-\nabla\omega|)+\sigma\int_{B_R}P(\mu+|\nabla\omega|)+\frac{C}{P(r-\rho)}\int_{B_R}P(|u|)\\
&\leq&\epsilon\int_{B_R}P(|\nabla\upsilon-\nabla\omega|)+\sigma\int_{B_R}P(|\nabla\omega|)+\frac{C}{P(r-\rho)}\int_{B_R}P(|u|)+C\, R^n\, P(\mu),
\end{eqnarray*}
where $\sigma>0$ will be chosen later. We continue to estimate $K_3$. According to the properties of $\eta$ and \eqref{young1}, it yields that
\begin{eqnarray*}
K_3&\leq&\int_{B_R\,\backslash\,B_r}\left|A(x,0)\right|\,\left|\nabla\upsilon-\nabla\omega\right|\\
&&+\,\,\int_{B_R\,\cap\,(B_r\backslash B_\rho)}\left|A(x,\tilde{P}(\eta)\nabla u)-P(\eta)A(x,\nabla u)\right|\,\left|\nabla\upsilon-\nabla\omega\right|\\
&\leq&\epsilon\int_{B_R}P(|\nabla\upsilon-\nabla\omega|)+C\int_{B_R}\tilde{P}\left(\left|A(x,0)\right|\right)\\
&&+\,\,C\int_{B_R\,\cap\,(B_r\backslash B_\rho)}\tilde{P}\left[\,\left|A(x,\tilde{P}(\eta)\nabla u)-P(\eta)A(x,\nabla u)\right|\,\right].
\end{eqnarray*}
Since $\tilde{P}$ is an $N$-function, the assumption (A3) gives us that 
\begin{eqnarray*}
K_3\!&\!\leq\!&\!\epsilon\int_{B_R}P(|\nabla\upsilon-\nabla\omega|)\,+\,C\int_{B_R}\tilde{P}\left[\frac{P(\mu)}{\mu}\right]\\
\!&\!\!&\!+\,C\int_{B_R\,\cap\,(B_r\backslash B_\rho)}\tilde{P}\left[\,\left|A(x,\tilde{P}(\eta)\nabla u)-A(x,\nabla u)\right|+\left|A(x,\nabla u)-P(\eta)A(x,\nabla u)\right|\,\right]\\
\!&\!\leq\!&\!\epsilon\int_{B_R}P(|\nabla\upsilon-\nabla\omega|)\,+\,C\,R^n\,P(\mu)\\
\!&\!\!&\!+\,\,C\int_{B_R\,\cap\,(B_r\backslash B_\rho)}\tilde{P}\left[\,\left|A(x,\tilde{P}(\eta)\nabla u)-A(x,\nabla u)\right|\,\right]\\
\!&\!\!&\!+\,\,C\int_{B_R\,\cap\,(B_r\backslash B_\rho)}\tilde{P}\left[\,\big|A(x,\nabla u)-P(\eta)A(x,\nabla u)\big|\,\right].
\end{eqnarray*}
With the help of (A2), it yields that
\begin{eqnarray*}
K_3&\leq&\epsilon\int_{B_R}P(|\nabla\upsilon-\nabla\omega|)\,+\,C\,R^n\,P(\mu)\\
&&+\,C\int_{B_R\,\cap\,(B_r\backslash B_\rho)}\tilde{P}\left[\,\frac{P\left[\left(\mu^2+|\nabla u|^2+\big|\tilde{P}(\eta)\nabla u\big|^2\right)^{\frac{1}{2}}\right]}{\mu^2+|\nabla u|^2+\big|\tilde{P}(\eta)\nabla u\big|^2}\,\left|\tilde{P}(\eta)\nabla u-\nabla u\right|\,\right]\\
&&+\,C\int_{B_R\,\cap\,(B_r\backslash B_\rho)}\tilde{P}\left[\,\big|1-P(\eta)\big|\,\frac{P\left[\left(\mu^2+|\nabla u|^2\right)^{\frac{1}{2}}\right]}{\left(\mu^2+|\nabla u|^2\right)^{\frac{1}{2}}}\,\right].
\end{eqnarray*}
Due to the definition of $N$-function and assumption (P1), we have
\begin{eqnarray*}
K_3&\leq&\epsilon\int_{B_R}P(|\nabla\upsilon-\nabla\omega|)\,+\,C\,R^n\,P(\mu)\\
&&+\,\,C\int_{B_R\,\cap\,(B_r\backslash B_\rho)}\tilde{P}\left[\,\frac{P(\mu)}{\mu}+\frac{P(|\nabla u|)}{|\nabla u|}\,\right]\\
&&+\,\,C\int_{B_R\,\cap\,(B_r\backslash B_\rho)}\tilde{P}\left[\,\frac{P\left[\,\big|\tilde{P}(\eta)\nabla u\big|\,\right]}{\big|\tilde{P}(\eta)\nabla u\big|^2}\,|\nabla u|+\frac{P\left[\left(\mu^2+|\nabla u|^2\right)^{\frac{1}{2}}\right]}{\mu^2+|\nabla u|^2}\,|\nabla u|\,\right]\\
&\leq&\epsilon\int_{B_R}P(|\nabla\upsilon-\nabla\omega|)\,+\,C\,R^n\,P(\mu)\\
&&+\,\,C\int_{B_R\,\cap\,(B_r\backslash B_\rho)}\tilde{P}\left[\frac{P(\mu)}{\mu}\right]+C\int_{B_R\,\cap\,(B_r\backslash B_\rho)}\tilde{P}\left[\frac{P(|\nabla u|)}{|\nabla u|}\right]\\
&&+\,\,C\int_{B_R\,\cap\,(B_r\backslash B_\rho)}\tilde{P}\left[\,\frac{P(|\nabla u|)}{|\nabla u|}+\frac{P\left[(\mu^2+|\nabla u|^2)^{\frac{1}{2}}\right]}{(\mu^2+|\nabla u|^2)^{\frac{1}{2}}}\,\right].
\end{eqnarray*}
By simplification, we obtain an estimate for $K_{3}$ as
\begin{eqnarray*}
K_3&\leq&\epsilon\int_{B_R}P(|\nabla\upsilon-\nabla\omega|)+C\int_{B_R}P(|\nabla u|)\,\chi_{B_r\backslash B_\rho}+C\,R^n\,P(\mu).
\end{eqnarray*}
For $K_4$, using \eqref{young1} again, one derives that 
\begin{eqnarray*}
K_4&\leq&\int_{B_R}P(\eta)\,\frac{P(|F|)}{|F|}\,|\nabla\upsilon-\nabla\omega|\,+\,\int_{B_R}P'(\eta)\,\left|\nabla\eta\right|\,\frac{P(|F|)}{|F|}\,|\upsilon-\omega|\\
&\leq&\epsilon\int_{B_R}P(|\nabla\upsilon-\nabla\omega|)\,+\,C\int_{B_R}P(|F|)+C\int_{B_R}\frac{R}{r-\rho}\,\frac{P(|F|)}{|F|}\,\frac{|\upsilon-\omega|}{R}.
\end{eqnarray*}
By the Assumption (P1) with $P(1)=1$, we get $t\leq \frac{P(t)}{t}$ if $t\geq 1$. Then we obtain 
\begin{eqnarray*}
&&\int_{B_R}\frac{R}{r-\rho}\,\frac{P(|F|)}{|F|}\,\frac{|\upsilon-\omega|}{R}\\
&\leq&\int_{B_R\,\cap\,\left\{\frac{R}{r-\rho}\,\leq\,1\right\}}\frac{P(|F|)}{|F|}\,\frac{|\upsilon-\omega|}{R}+\int_{B_R\,\cap\,\left\{\frac{R}{r-\rho}\,>\,1\right\}}\frac{P\left(\frac{R}{r-\rho}\right)}{\frac{R}{r-\rho}}\,\frac{P(|F|)}{|F|}\,\frac{|\upsilon-\omega|}{R}\\
&\leq&\epsilon\int_{B_R}P\left(\frac{|\upsilon-\omega|}{R}\right)\,+\,C\left[P\left(\frac{R}{r-\rho}\right)+1\right]\int_{B_R}P(|F|).
\end{eqnarray*}
Applying a modified version of Poincar\'{e}-Wirtinger inequality with $R<\frac{k_0d}{\delta}$, one obtains that
\begin{eqnarray*}
K_4&\leq&\epsilon\int_{B_R}P(|\nabla\upsilon-\nabla\omega|)\,+\,C\left[P\left(\frac{R}{r-\rho}\right)+1\right]\int_{B_R}P(|F|).
\end{eqnarray*}
To estimate $K_{5}$, we use the hypothesis (A3) and the property of $\nabla\eta$ to obtain that 
\begin{eqnarray*}
K_5&\leq&C\int_{B_R\,\cap\,(B_r\backslash B_\rho)}\frac{R}{r-\rho}\,\frac{P\left[\left(\mu^2+|\nabla u|^2\right)^{\frac{1}{2}}\right]}{\left(\mu^2+|\nabla u|^2\right)^{\frac{1}{2}}}\,\frac{|\upsilon-\omega|}{R}.
\end{eqnarray*}
In view of \eqref{young1} and the modified version of Poincar\'{e}-Wirtinger inequality, we acquire
\begin{eqnarray*}
K_5&\leq&C\int_{B_R\,\cap\,(B_r\backslash B_\rho)\,\cap\,\left\{\frac{R}{r-\rho}\,\leq\,1\right\}}\frac{P\left[\left(\mu^2+|\nabla u|^2\right)^{\frac{1}{2}}\right]}{\left(\mu^2+|\nabla u|^2\right)^{\frac{1}{2}}}\,\frac{|\upsilon-\omega|}{R}\\
&&+\,\,C\int_{B_R\,\cap\,(B_r\backslash B_\rho)\,\cap\,\left\{\frac{R}{r-\rho}\,>\,1\right\}}\frac{P\left(\frac{R}{r-\rho}\right)}{\frac{R}{r-\rho}}\,\frac{P\left[\left(\mu^2+|\nabla u|^2\right)^{\frac{1}{2}}\right]}{\left(\mu^2+|\nabla u|^2\right)^{\frac{1}{2}}}\,\frac{|\upsilon-\omega|}{R}\\[0.1cm]
&\leq&\epsilon\int_{B_R}P\left(\frac{|\upsilon-\omega|}{R}\right)+C\left[P\left(\frac{R}{r-\rho}\right)+1\right]\int_{B_R\,\cap\,(B_r\backslash B_\rho)}P\left[\left(\mu^2+|\nabla u|^2\right)^{\frac{1}{2}}\right]\\
&\leq&\epsilon\int_{B_R}P(|\nabla\upsilon-\nabla\omega|)\,+\,C\,P\left(\frac{R}{r-\rho}\right)\int_{B_R}P(|\nabla u|)\\
&&+\,\,C\int_{B_R}P(|\nabla u|)\,\chi_{B_r\backslash B_\rho}\,+\,C\left[P\left(\frac{R}{r-\rho}\right)+1\right]R^n\,P(\mu).
\end{eqnarray*}
Combining the estimates of $K_i$ ($0\leq i\leq 5$), we conclude that
\begin{eqnarray*}
&&\upsilon\int_{B_R}P(|\nabla\upsilon-\nabla\omega|)\\
&\leq&5\,\epsilon\int_{B_R}P(|\nabla\upsilon-\nabla\omega|)\,+\,\sigma\int_{B_R}P(|\nabla\omega|)\\
&&+\,\,C\left(\widetilde{T\circ P^{-1}}\right)^{-1}\left[\int_{B_R}V(x,B_R)\right]\cdot\left(P\circ T^{-1}\right)\left[\int_{B_R}T\left[\left(\mu^2+|\nabla\omega|^2\right)^{\frac{1}{2}}\right]\right]\\
&&+\,\,\frac{C}{P(r-\rho)}\int_{B_R}P(|u|)\,+\,C\int_{B_R}P(|\nabla u|)\,\chi_{B_r\backslash B_\rho}\,+\,C\,P\left(\frac{R}{r-\rho}\right)\int_{B_R}P(|\nabla u|)\\
&&+\,\,C\left[P\left(\frac{R}{r-\rho}\right)+1\right]\int_{B_R}P(|F|)\,+\,C\left[P\left(\frac{R}{r-\rho}\right)+1\right]R^n\,P(\mu).
\end{eqnarray*}
If we choose $\epsilon=\frac{\upsilon}{10}$, then the first term on the right hand side of inequality can be absorbed by the left hand side, and so 
\begin{eqnarray}\label{PDv-Dw}
&&\int_{B_R}P(|\nabla\upsilon-\nabla\omega|)\\
&\leq&\sigma\int_{B_R}P(|\nabla\omega|)\nonumber\\
&&+\,\,C\left(\widetilde{T\circ P^{-1}}\right)^{-1}\left[\int_{B_R}V(x,B_R)\right]\cdot\left(P\circ T^{-1}\right)\left[\int_{B_R}T\left[\left(\mu^2+|\nabla\omega|^2\right)^{\frac{1}{2}}\right]\right]\nonumber\\
&&+\,\,\frac{C}{P(r-\rho)}\int_{B_R}P(|u|)\,+\,C\int_{B_R}P(|\nabla u|)\,\chi_{B_r\backslash B_\rho}+\,C\,P\left(\frac{R}{r-\rho}\right)\int_{B_R}P(|\nabla u|)\nonumber\\
&&+\,\,C\left[P\left(\frac{R}{r-\rho}\right)+1\right]\int_{B_R}P(|F|)\,+\,C\left[P\left(\frac{R}{r-\rho}\right)+1\right]R^n\,P(\mu).\nonumber
\end{eqnarray}
\textbf{Step 3.} Considering the ball $B_{\delta R}=B_{\delta R}(y)$, it can be verified that there exists $\kappa>0$ such that
\begin{eqnarray}\label{equ4}
\fint_{B_{\delta R}}\!P(|\nabla\omega\!-\!(\nabla\omega)_{B_{\delta R}}|)\!\!&\!\!\leq\!\!&\!\! C\fint_{B_{\delta R}}\!P(|\nabla\upsilon\!-\!(\nabla\upsilon)_{B_{\delta R}}|)+C\delta^{-n}\fint_{B_{R}}\!P(|\nabla\omega\!-\!\nabla\upsilon|),\\ \label{equ5}
\fint_{B_{\delta R}}\!P(|\nabla\upsilon\!-\!(\nabla\upsilon)_{B_{\delta R}}|)\!\!&\!\!\leq\!\!&\!\! C\,P(\delta^\kappa)\fint_{B_R}\!P(\mu+|\nabla\omega|).
\end{eqnarray}
Applying \eqref{PDv-Dw}, \eqref{equ4}, and \eqref{equ5}, we continue estimating  
\begin{eqnarray*}
&&\fint_{B_{\delta R}}P(|\nabla\omega-(\nabla\omega)_{B_{\delta R}}|)\\
\!&\!\leq\!&\!C\left[P(\delta^\kappa)+\sigma\,\delta^{-n}\right]\fint_{B_R}P(|\nabla\omega|)\nonumber\\
&&+\,\,C\,\delta^{-n}\left(\widetilde{T\circ P^{-1}}\right)^{-1}\left[\fint_{B_R}V(x,B_R)\right]\cdot\left(P\circ T^{-1}\right)\left[\fint_{B_R}T\left[\left(\mu^2+|\nabla\omega|^2\right)^{\frac{1}{2}}\right]\right]\nonumber\\
&&+\,\frac{C\,\delta^{-n}}{P(r-\rho)}\fint_{B_R}P(|u|)+C\,\delta^{-n}\fint_{B_R}P(|\nabla u|)\,\chi_{B_r\backslash B_\rho}+C\,\delta^{-n}\,P\left(\frac{R}{r-\rho}\right)\fint_{B_R}P(|\nabla u|)\nonumber\\
&&+\,\,C\,\delta^{-n}\left[P\left(\frac{R}{r-\rho}\right)+1\right]\fint_{B_R}P(|F|)\nonumber\\
&&+\,\,C\,\delta^{-n}\left[P\left(\frac{R}{r-\rho}\right)+1\right]P(\mu)\,+\,C\,P(\delta^{\kappa})\,P(\mu). \nonumber 
\end{eqnarray*}
By the classical regularity theory, one gets
\begin{equation*}
\fint_{B_R}P(|\nabla u|)\,\leq\,\frac{C}{P(R)}\fint_{B_{2R}}P(|u|)\,+\,C\fint_{B_{2R}}P(|F|). 
\end{equation*}
Therefore, we have
\begin{eqnarray*}
&&\fint_{B_{\delta R}}P(|\nabla\omega-(\nabla\omega)_{B_{\delta R}}|)\\
&\leq& C\left[P(\delta^\kappa)+\sigma\,\delta^{-n}\right]\fint_{B_{R}}P(|\nabla\omega|)\\
&&+\,\,C\,\delta^{-n}\left(\widetilde{T\circ P^{-1}}\right)^{-1}\left[\fint_{B_R}V(x,B_R)\right]\cdot\left(P\circ T^{-1}\right)\left[\fint_{B_R}T\left[\left(\mu^2+|\nabla\omega|^2\right)^{\frac{1}{2}}\right]\right]\\
&&+\,\,\frac{C\,\delta^{-n}}{P(r-\rho)}\fint_{B_{2R}}P(|u|)\,+\,C\,\delta^{-n}\fint_{B_{R}}P(|\nabla u|)\,\chi_{B_r\backslash B_\rho}\\
&&+\,\,C\,\delta^{-n}\left[P\left(\frac{R}{r-\rho}\right)+1\right]\fint_{B_{2R}}P(|F|)\\
&&+\,\,C\,\delta^{-n}\left[P\left(\frac{R}{r-\rho}\right)+1\right]P(\mu)\,+\,C\,P(\delta^{\kappa})\,P(\mu).
\end{eqnarray*}
Since both $B_{2R}\subset\lambda B_{d}$, we conclude that
\begin{eqnarray*}
&&\fint_{B_{\delta R}}P(|\nabla\omega-(\nabla\omega)_{B_{\delta R}}|)\\
&\leq&CP(\delta^\kappa)\fint_{B_{R}}P(|\nabla\omega|)\\
&&+\,\,C\,\delta^{-n}\left(\widetilde{T\circ P^{-1}}\right)^{-1}\left[\fint_{B_R}V(x,B_R)\right]\cdot\left(P\circ T^{-1}\right)\left[\fint_{B_R}T\left[\left(\mu^2+|\nabla\omega|^2\right)^{\frac{1}{2}}\right]\right]\\
&&+\,\,\frac{C\,\delta^{-n}}{P(r-\rho)}\fint_{B_{2R}}P(|\,u\,\chi_{\lambda B_{d}}|)++C\,\delta^{-n}\fint_{B_{R}}P(|\nabla u|)\,\chi_{B_r\backslash B_\rho}\\
&&+\,\,C\,\delta^{-n}\left[P\left(\frac{R}{r-\rho}\right)+1\right]\fint_{B_{2R}}P(|\,F\,\chi_{\lambda B_{d}}|)\\
&&+\,\,C\,\delta^{-n}\left[P\left(\frac{R}{r-\rho}\right)+1\right]P(\mu)\,+\,C\,P(\delta^{\kappa})\,P(\mu),
\end{eqnarray*}
where we choose $\sigma=\frac{P(\delta^\kappa)}{\delta^{-n}}$. \\
\textbf{Step 4.} According to Definition \ref{sharp} and taking supremum over every $R\in(0,\frac{k_0d}{\delta})$, it yields that
\begin{eqnarray*}
P\left(\mathcal{M}_{P,k_0d}^\#(|\nabla\omega|)(y)\right)\!\!&\!\!\leq\!\!&\!\!C\,P(\delta^\kappa)P\left(\mathcal{M}_P(|\nabla\omega|)(y)\right)\\
&&+\,\,C\,\delta^{-n}\sup_{0<R<\frac{k_0d}{\delta}}\left(\widetilde{T\circ P^{-1}}\right)^{-1}\left[\fint_{B_R}V(x,B_R)\right]\cdot P\left(\mathcal{M}_T(|\nabla\omega|)(y)\right)\\
&&+\,\frac{C\,\delta^{-n}}{P(r-\rho)}P\left(\mathcal{M}_P(|\,u\,\chi_{\lambda B_{d}}|)(y)\right)+C\,\delta^{-n}P(\mathcal{M}_P(\,|\nabla u|\,\chi_{B_r\backslash B_\rho})(y))\\
&&+\,\,C\,\delta^{-n}\left[P\left(\frac{R}{r-\rho}\right)+1\right]P\left(\mathcal{M}_P(|\,F\,\chi_{\lambda B_{d}}|)(y)\right)\\
&&+\,\,C\,\delta^{-n}\left[P\left(\frac{R}{r-\rho}\right)+1\right]P(\mu)\,+\,C\,P(\delta^{\kappa})\,P(\mu). 
\end{eqnarray*}
Since $Q\circ P^{-1}$ is an $N$-function, we deduce that
\begin{eqnarray*}
&&Q\left(\mathcal{M}_{P,k_0d}^\#(|\nabla\omega|)(y)\right)\\
&=&(Q\circ P^{-1}\circ P)\left(\mathcal{M}_{P,k_0d}^\#(|\nabla\omega|)(y)\right)\\[0.1cm]
&\leq& C\,Q(\delta^\kappa)\,Q\left(\mathcal{M}_P(|\nabla\omega|)(y)\right)\\
&&+\,\,C\,\delta^{-n}\sup_{0<R<\frac{k_0d}{\delta}}\left(\widetilde{T\circ P^{-1}}\right)^{-1}\left[\fint_{B_R}V(x,B_R)\right]\cdot Q\left(\mathcal{M}_T(|\nabla\omega|)(y)\right)\\
&&+\,\,\frac{C\,\delta^{-n}}{Q(r-\rho)}Q\left(\mathcal{M}_P(|\,u\,\chi_{\lambda B_{d}}|)(y)\right)+C\,\delta^{-n}Q(\mathcal{M}_P(\,|\nabla u|\,\chi_{B_r\backslash B_\rho})(y))\\
&&+\,\,C\,\delta^{-n}\left[Q\left(\frac{R}{r-\rho}\right)+1\right]Q\left(\mathcal{M}_P(|\,F\,\chi_{\lambda B_{d}}|)(y)\right)\\
&&+\,\,C\,\delta^{-n}\left[Q\left(\frac{R}{r-\rho}\right)+1\right]Q(\mu)\,+\,C\,Q(\delta^{\kappa})\,Q(\mu).
\end{eqnarray*}
Integrating with respect to $y$ over $k_0B_d$, we acquire that
\begin{eqnarray*}
&&\int_{k_0B_d}Q\left(\mathcal{M}_{P,k_0d}^\#(|\nabla\omega|)\right)\\
&\leq&C\,Q(\delta^\kappa)\int_{\mathbb{R}^n}Q\left(\mathcal{M}_P(|\nabla\omega|)\right)\\
&&\,\,C\,\delta^{-n}\sup_{y\in k_0B_d}\,\,\sup_{0<R<\frac{k_0d}{\delta}}\left(\widetilde{T\circ P^{-1}}\right)^{-1}\left[\fint_{B_R}V(x,B_R)\right]\cdot\int_{\mathbb{R}^n}Q\left(\mathcal{M}_T(|\nabla\omega|)\right)\\
&&+\,\,\frac{C\,\delta^{-n}}{Q(r-\rho)}\int_{\mathbb{R}^n}Q\left(\mathcal{M}_P(|\,u\,\chi_{\lambda B_d}|)\right)+C\,\delta^{-n}\int_{\mathbb{R}^n}Q\left(\mathcal{M}_P(\,|\nabla u|\,\chi_{B_r\backslash B_\rho})\right)\\
&&+\,\,C\,\delta^{-n}\left[Q\left(\frac{R}{r-\rho}\right)+1\right]\int_{\mathbb{R}^n}Q\left(\mathcal{M}_P(|\,F\,\chi_{\lambda B_{d}}|)\right)\\
&&+\,\,C\,|k_0B_d|\,\delta^{-n}\left[Q\left(\frac{R}{r-\rho}\right)+1\right]Q(\mu)\,+\,C\,|k_0B_d|\,Q(\delta^{\kappa})\,Q(\mu).
\end{eqnarray*}
By \eqref{M-Msharp}, it follows that 
\begin{eqnarray}\label{equ7}
\int_{B_d}Q(|\nabla\omega|)&\!\!\leq\!\!&C\int_{k_0B_d}Q\left(\mathcal{M}_{P,k_0d}^\#(|\nabla\omega|)\right)\nonumber\\
&\leq&C\,Q(\delta^\kappa)\int_{B_d}Q(|\nabla\omega|) \nonumber \\
&&+\,\,C\,\delta^{-n}\sup_{y\in k_0B_d}\,\,\sup_{0<R<\frac{k_0d}{\delta}}\left(\widetilde{T\circ P^{-1}}\right)^{-1}\left[\fint_{B_R}V(x,B_R)\right]\cdot\int_{B_d}Q(|\nabla\omega|)\nonumber\\
&&+\,\,\frac{C\,\delta^{-n}}{Q(r-\rho)}\int_{\mathbb{R}^n}Q(|\,u\,\chi_{\lambda B_d}|)+C\,\delta^{-n}\int_{\mathbb{R}^n}Q(\,|\nabla u|\,\chi_{B_r\backslash B_\rho})\nonumber\\
&&+\,\,C\,\delta^{-n}\left[Q\left(\frac{R}{r-\rho}\right)+1\right]\int_{\mathbb{R}^n}Q(|\,F\,\chi_{\lambda B_{d}}|)\nonumber\\
&&+\,\,C\,|k_0B_d|\,\delta^{-n}\left[Q\left(\frac{R}{r-\rho}\right)+1\right]Q(\mu)\,+\,C\,|k_0B_d|\,Q(\delta^{\kappa})\,Q(\mu).
\end{eqnarray}
Our aim consists of inserting the two terms involving $\nabla\omega$ on the right hand side into the term on the left hand side, by making relevant coefficients as small as possible. We choose $\delta$ such that
\begin{equation*}
C\,Q(\delta^\kappa)=\frac{1}{4}\ \Longleftrightarrow\ \delta=\left[Q^{-1}\left(\frac{1}{4C}\right)\right]^{\frac{1}{\kappa}}.
\end{equation*}
Note that this choice of $\delta>0$ is independent of $d$. Therefore, \eqref{equ7} reads as 
\begin{eqnarray}\label{equ17}
\int_{B_d}Q(|\nabla\omega|) &\leq& C\sup_{y\in k_0B_d}\,\,\sup_{0<R<\frac{k_0d}{\delta}}\left(\widetilde{T\circ P^{-1}}\right)^{-1}\left[\fint_{B_R}V(x,B_R)\right]\cdot\int_{B_d}Q(|\nabla\omega|)\nonumber\\
&&+\,\,\frac{C}{Q(r-\rho)}\int_{\mathbb{R}^n}Q(|\,u\,\chi_{\lambda B_d}|)+C\int_{\mathbb{R}^n}Q(\,|\nabla u|\,\chi_{B_r\backslash B_\rho})\nonumber\\
&&+\,\,C\left[Q\left(\frac{k_0d}{r-\rho}\right)+1\right]\int_{\mathbb{R}^n}Q(|\,F\,\chi_{\lambda B_{d}}|) \nonumber \\
&&+\,\,C\,|k_0B_d|\left[Q\left(\frac{k_0d}{r-\rho}\right)+1\right]Q(\mu).
\end{eqnarray}
Moreover, if $k_0d<\frac{{\operatorname{dist}}(x_0,\,\partial\Omega)}{2}$, then we get that
\begin{eqnarray*}
&&\sup_{y\in k_0B_d}\,\,\sup_{0<R<\frac{k_0d_0}{\delta}}\left(\widetilde{T\circ P^{-1}}\right)^{-1}\left[\fint_{B_R}V(x,B_{R})\right]\\
&\leq&\sup_{y\in B_{\frac{{\operatorname{dist}}(x_0,\,\Omega)}{2}}(x_0)}\,\,\sup_{0<R<\frac{k_0d_0}{\delta}}\left(\widetilde{T\circ P^{-1}}\right)^{-1}\left[\fint_{B_R}V(x,B_R)\right],
\end{eqnarray*}
and hence 
\begin{equation}
\lim_{d\rightarrow0}\,\,\sup_{y\in B_{\frac{{\operatorname{dist}}(x_0,\,\partial\Omega)}{2}}(x_0)}\,\,\sup_{0<R<\frac{k_0d_0}{\delta}}\left(\widetilde{T\circ P^{-1}}\right)^{-1}\left[\fint_{B_R}V(x,B_R)\right]=0.\nonumber
\end{equation}
In fact, $d>0$ can be selected small enough such that
\begin{equation}
C\sup_{y\in k_0B_d}\,\,\sup_{0<R<\frac{k_0d_0}{\delta}}\left(\widetilde{T\circ P^{-1}}\right)^{-1}\left[\fint_{B_R}V(x,B_R)\right]<\frac{1}{4}.\nonumber
\end{equation}
This allows us to insert the remaining term involving $\nabla\omega$ into the left hand side of \eqref{equ17}. It follows that 
\begin{eqnarray*}
\int_{B_d}Q(|\nabla\omega|)&\leq&\frac{C}{Q(r-\rho)}\int_{\lambda B_d}Q(|u|)+C\int_{B_r\backslash B_\rho}Q(|\nabla u|)\\
&&+\,\,C\left[Q\left(\frac{k_0d}{r-\rho}\right)+1\right]\int_{\lambda B_d}Q(|F|)\\
&&+\,\,C\,|\lambda B_d|\left[Q\left(\frac{k_0d}{r-\rho}\right)+1\right]Q(\mu).
\end{eqnarray*}
Since $\omega=\tilde{P}(\eta)u$ with $\chi_{B_{\rho}}\leq\eta\leq\chi_{B_r}$ and $\rho<d$, then we obtain 
\begin{eqnarray*}
\int_{B_{\rho}}Q(|\nabla u|)\!&\!\leq&\!\!C\int_{B_d}Q(|\nabla\omega|)\\
&\!\leq\!&\!\frac{C}{Q(r-\rho)}\int_{\lambda B_d}Q(|u|)+C'\int_{B_r\backslash B_\rho}Q(|\nabla u|)\\
&&+\,C\left[Q\left(\frac{d}{r-\rho}\right)+1\right]\int_{\lambda B_d}Q(|F|)+C\,|\lambda B_d|\left[Q\left(\frac{d}{r-\rho}\right)+1\right]Q(\mu).
\end{eqnarray*}
Filling the hole, that is, adding to both sides of previous inequality the quantity $C'\int_{B_\rho}Q(|\nabla u|)$, we acquire
\begin{eqnarray*}
\int_{B_\rho}Q(|\nabla u|)&\leq&\nu\int_{B_r}Q(|\nabla u|)+\frac{C}{Q(r-\rho)}\int_{\lambda B_d}Q(|u|)\\
&&+\,\,C\left[Q\left(\frac{d}{r-\rho}\right)+1\right]\int_{\lambda B_d}Q(|F|)+C\,|\lambda B_d|\left[Q\left(\frac{d}{r-\rho}\right)+1\right]Q(\mu),
\end{eqnarray*}
where $0<\nu<1$. Via the choice of $d$ and $\rho$, one applies Lemma \ref{hole} to derive that 
\begin{eqnarray}
\int_{B_\frac{d}{2}}Q(|\nabla u|)&\leq&\frac{C}{Q(r-\rho)}\int_{\lambda B_d}Q(|u|)+C\left[Q\left(\frac{d}{r-\rho}\right)+1\right]\int_{\lambda B_d}Q(|F|)\nonumber\\
&&+\,\,C\,|\lambda B_d|\left[Q\left(\frac{d}{r-\rho}\right)+1\right]Q(\mu).
\end{eqnarray}
Since above estimates are valid for arbitrary radii $\frac{d}{2}<\rho<r<d$, we conclude that
\begin{equation}
\fint_{\frac{1}{2}B_d}Q(|\nabla u|)\leq C\left[Q(\mu)+\fint_{\lambda B_{d}}Q(|F|)+\frac{1}{Q(d)}\fint_{\lambda B_{d}}Q(|u|)\right].\nonumber
\end{equation}
This completes the proof of Proposition \ref{Osc-estimate}. 
\end{proof}

\subsection{Higher Integrability}
The higher integrability estimate for Laplace and $p$-Laplace equation is well known (see $\cite{9}$ and $\cite{10}$). In this section, we shall prove the following proposition.  
\begin{proposition}\label{higher int}
Assume that {\rm{(A1)}-\rm{(A3)}}, {\rm{(P1)}} hold, and $x\mapsto A(x,\xi)$ is locally uniformly in VMO. If u is a weak solution to \eqref{equ} with $F\in L^Q_{{\rm{loc}}}(\Omega)$, then $\nabla u\in L^Q_{{\rm{loc}}}(\Omega)$, where $Q(t)$ is an $N$-function satisfying that $Q\circ P^{-1}$ is also an $N$-function. Moreover, there exists a constant $\lambda>1$ such that 
\begin{equation}\label{equ8}
\fint_BQ(|\nabla u|)\leq C\left(1+\fint_{\lambda B}Q(|u|)+\fint_{\lambda B}Q(|F|)\right)
\end{equation}
holds for any ball B such that $\lambda B\subset\Omega$. 
\end{proposition}
\begin{proof}
We fix a ball $B_R(x_0)\Subset\Omega$ with $0<R<\lambda\,d_0$ where $\lambda$ and $d_0$ are the ones determined in Proposition \ref{Osc-estimate}. By taking a smooth $\phi\in C_0^\infty(B_1(0))$ with $\phi\geq0$ and$\int_{B_1(0)}\phi=1$, and considering the mollifiers $(\phi_\epsilon)_{\epsilon>0}$, we let
\begin{equation}
A_\epsilon(x,\xi)=A(\cdot, \xi)\ast\phi_\epsilon(x)=\int_{B_1}\phi(y)\,A(x-\epsilon y,\xi)\operatorname{d}\! y \nonumber
\end{equation}
and $F_\epsilon=F\ast\phi_\epsilon$, where $0<\epsilon<\operatorname{dist}(B_R,\Omega)$. Then the assumptions (A1)-(A3) imply that 
\begin{eqnarray*}
&\textrm{(H1)}&\left[A_\epsilon(x,\xi)-A_\epsilon(x,\eta)\right]\cdot(\xi-\eta)\geq\upsilon\frac{P\left[\left(\mu^2+|\xi|^2+|\eta|^2\right)^{\frac{1}{2}}\right]}{\mu^2+|\xi|^2+|\eta|^2}|\xi-\eta|^2\,,\\
&\textrm{(H2)}&\left|A_\epsilon(x,\xi)-A_\epsilon(x,\eta)\right|\leq L\frac{P\left[\left(\mu^2+|\xi|^{2}+|\eta|^2\right)^{\frac{1}{2}}\right]}{\mu^2+|\xi|^{2}+|\eta|^{2}}|\xi-\eta|\,,\\
&\textrm{(H3)}&|A_\epsilon(x,\xi)|\leq k\frac{P\left[\left(\mu^2+|\xi|^2\right)^{\frac{1}{2}}\right]}{\left(\mu^2+|\xi|^2\right)^{\frac{1}{2}}}
\end{eqnarray*}
for almost every $x\in\Omega$ and for all $\xi$, $\eta\in\mathbb{R}^n$. Furthermore, we let
\begin{equation}
V_\epsilon(x,B)=\sup_{\xi\in\mathbb{R}^{n}}\frac{|A_\epsilon(x,\xi)-A_{\epsilon,B}(\xi)|}{\frac{P\left[\left(\mu^2+|\xi|^2\right)^{\frac{1}{2}}\right]}{\left(\mu^2+|\xi|^2\right)^{\frac{1}{2}}}}, \nonumber
\end{equation}
where
\begin{equation}
A_{\epsilon,B}(\xi)=\fint_{B}A_\epsilon(y,\xi)\operatorname{d}\!y. \nonumber
\end{equation}
Because $x\rightarrow A_\epsilon(x,\xi)$ is $C^\infty$ smooth, we have that
\begin{equation}
\textrm{(H4)}\quad \lim_{r\rightarrow0}\,\,\sup_{r(B)<r}\,\,\sup_{c(B)\in K\subset\Omega}\,\fint_B V_\epsilon(x,B)\operatorname{d}\!x=0.\nonumber
\end{equation}
In order to complete the proof of this proposition, we claim that, if $A_\epsilon(x,\nabla u)\in L^{\tilde{P}}(B_R)$ and $F\in L^Q(B_R)$, then 
\begin{eqnarray*}
A_\epsilon(x,\nabla u)&\rightarrow& A(x,\nabla u)\quad\text{strongly\ in\ } L^{\tilde{P}}(B_R),\\
F_\epsilon&\rightarrow& F\quad\text{strongly\ in\ } L^{Q}(B_R). 
\end{eqnarray*}
Let $u$ $\in W^{1,P}_{\text{loc}}(\Omega)$ be a weak solution to \eqref{equ}, and $u_\epsilon\in W^{1,P}(B_R)$ a weak solution to the following Dirichlet problem
\begin{equation}\label{equ9}
\left\{
\begin{aligned}
&\operatorname{div}A_\epsilon(x,\nabla u_\epsilon)=\operatorname{div}\left[\frac{P(|F_\epsilon|)}{|F_\epsilon|^2}F_\epsilon\right]\ &\text{in}\ B_{R},\\
&u_\epsilon=u\ &\text{on}\ \partial B_{R}.
\end{aligned}
\right.
\end{equation}
By the classical theory, one derives that $\nabla u_\epsilon\in L^Q(B_{R})$. By taking $\phi=u_\epsilon-u$ as a test function in \eqref{equ9} and in \eqref{equ}, we obtain that 
\begin{equation}
\int_{B_R}A_\epsilon(x,\nabla u_\epsilon)\cdot(\nabla u-\nabla u_\epsilon)=\int_{B_R}\frac{P(|F_\epsilon|)}{|F_\epsilon|^2}F_\epsilon\cdot(\nabla u-\nabla u_\epsilon)\nonumber
\end{equation}
and
\begin{equation}
\int_{B_R}A(x,\nabla u)\cdot(\nabla u-\nabla u_\epsilon)=\int_{B_R}\frac{P(|F|)}{|F|^2}F\cdot(\nabla u-\nabla u_\epsilon).\nonumber
\end{equation}
By subtraction, we get that 
\begin{eqnarray*}
\int_{B_R}\left[A_\epsilon(x,\nabla u_\epsilon)-A(x,\nabla u)\right]\cdot(\nabla u-\nabla u_\epsilon)=\int_{B_R}\left[\frac{P(|F_\epsilon|)}{|F_\epsilon|^2}F_\epsilon-\frac{P(|F|)}{|F|^2}F\right]\cdot(\nabla u-\nabla u_\epsilon). 
\end{eqnarray*}
By using (H1), it yields that 
\begin{eqnarray*}
&&\upsilon\int_{B_R}\frac{P\left[\left(\mu^2+|\nabla u|^2+|\nabla u_\epsilon|^2\right)^{\frac{1}{2}}\right]}{\mu^2+|\nabla u|^2+|\nabla u_\epsilon|^2}|\nabla u-\nabla u_\epsilon|^2\\
&\leq&\int_{B_R}\left[A(x,\nabla u)-A_\epsilon(x,\nabla u)\right]\cdot(\nabla u-\nabla u_\epsilon)\\
&&+\,\,\int_{B_R}\left[\frac{P(|F_\epsilon|)}{|F_\epsilon|^2}F_\epsilon-\frac{P(|F|)}{|F|^2}F\right]\cdot(\nabla u-\nabla u_\epsilon). 
\end{eqnarray*}
According to the modified version of H\"{o}lder inequality, it is easy to get
\begin{eqnarray}\label{equ18}
&&\upsilon\int_{B_R}\frac{P\left[\left(\mu^2+|\nabla u|^2+|\nabla u_\epsilon|^2\right)^{\frac{1}{2}}\right]}{\mu^2+|\nabla u|^2+|\nabla u_\epsilon|^2}|\nabla u-\nabla u_\epsilon|^2\nonumber\\
&\!\!\leq\!\!&\tilde{P}^{-1}\left[\int_{B_R}\tilde{P}\left(|A(x,\nabla u)-A_\epsilon(x,\nabla u)|\right)\right]\cdot P^{-1}\left[\int_{B_R}P(|\nabla u-\nabla u_\epsilon|)\right]\nonumber\\
&&+\,\tilde{P}^{-1}\left[\int_{B_R}\tilde{P}\left(\left|\frac{P(|F_\epsilon|)}{|F_\epsilon|^2}F_\epsilon-\frac{P(|F|)}{|F|^2}F\right|\right)\right]\cdot P^{-1}\left[\int_{B_R}P(|\nabla u-\nabla u_\epsilon|)\right]. \nonumber 
\end{eqnarray}
With the help of \eqref{young1}, we deduce that 
\begin{equation*}
\int_{B_R}P(|\nabla u-\nabla u_\epsilon|)\leq C\int_{B_R}\tilde{P}(|A(x,\nabla u)-A_\epsilon(x,\nabla u)|)+\int_{B_R}\tilde{P}\left(\left|\frac{P(|F_\epsilon|)}{|F_\epsilon|^2}F_\epsilon-\frac{P(|F|)}{|F|^2}F\right|\right). 
\end{equation*}
Taking $\epsilon\rightarrow0$, we get that $u_\epsilon$ converges strongly to $u$ in $W^{1,P}(B_{R})$. Since $A_\epsilon$ satisfies assumptions (H1)-(H4) and $\nabla u_\epsilon\in L^Q(B_{R})$, we use the estimate of Proposition \ref{Osc-estimate} for $u_\epsilon$ and acquire
\begin{equation}\label{equ10}
\int_{B_{\rho}}Q(|\nabla u_\epsilon|)\leq C\left(Q(\mu)\,|B_{\lambda\rho}|+\frac{1}{Q(\lambda\rho)}\int_{B_{\lambda\rho}}Q(|u_\epsilon|)+\int_{B_{\lambda\rho}}Q(|F_\epsilon|)\right)
\end{equation}
for every $\rho>0$ such that $B_{\lambda\rho}\subset B_R$. Then there exist a sequence of $N$-functions $\{Q_{i}\}_{i\in\mathbb{N}}$ satisfying that  
\begin{equation}
\left\{
\begin{aligned}
&Q_0=Q\,,\\
&Q_i=\left(Q_i\circ Q^{-1}_{i+1}\right)\circ Q_{i+1}\,,
&(i\in\mathbb{N})\nonumber
\end{aligned}
\right.
\end{equation}
where $Q_i\circ Q^{-1}_{i+1}$ is also an $N$-functions. We note that there exists $h\in\mathbb{N}$ such that $L^P\subset L^{Q_h}$. We choose $\rho$ small such that $\lambda^h\rho<R$ and let $r_i=\lambda^i\rho$. Since $F\in L^Q(B_R)$, we obtain $F\in L^{Q_i}(B_R)$ for every $i\in\mathbb{N}$, and so we estimate \eqref{equ10} as
\begin{eqnarray}\label{equ19}
\int_{B_{r_i}}Q_i(|\nabla u_\epsilon|)\!&\!\leq\!&\!\frac{C}{Q_i(r_{i+1})}\int_{B_{r_{i+1}}}Q_i(|u_\epsilon|)+C\int_{B_{r_{i+1}}}Q_i(|F_\epsilon|)+C\,Q_i(\mu)\,|B_{r_{i+1}}|\nonumber\\
\!&\!\leq\!&\!\frac{C}{Q_i(r_{i+1})}\int_{B_{r_{i+1}}}Q_{i+1}(|u_\epsilon|)+\frac{C}{Q_i(r_{i+1})}\int_{B_{r_{i+1}}}Q_{i+1}(|\nabla u_\epsilon|)\nonumber\\
\!&\!\!&\!+\,C\int_{B_{r_{i+1}}}Q_i(|F_\epsilon|)+C\,Q_i(\mu)\,|B_{r_{i+1}}|\nonumber\\
\!&\!\leq\!&\!\frac{C}{Q_i(r_{i+1})}\int_{B_{r_{i+1}}}Q_{i+1}(|u_\epsilon|)+C\int_{B_{r_{i+1}}}Q_i(|F_\epsilon|)+C\,Q_i(\mu)\,|B_{r_{i+1}}|\nonumber\\
\!&\!\!&\!+\,\frac{C}{Q_i(r_{i+1})\,Q_{i+1}(r_{i+2})}\int_{B_{r_{i+2}}}Q_{i+1}(|u_\epsilon|)+\frac{C}{Q_i(r_{i+1})}\int_{B_{r_{i+2}}}Q_{i+1}(|F_\epsilon|)\nonumber\\
\!&\!\!&\!+\,\frac{C}{Q_i(r_{i+1})}\,Q_{i+1}(\mu)\,|B_{r_{i+2}}| \nonumber \\
\!&\!\leq\!&\!\frac{C}{Q_i(r_{i+1})\,Q_{i+1}(r_{i+2})}\int_{B_{r_{i+2}}}Q_{i+1}(|u_\epsilon|)+\frac{C}{Q_i(r_{i+1})}\int_{B_{r_{i+2}}}Q_{i+1}(|F_\epsilon|)\nonumber\\
\!&\!\!&\!+\,C\int_{B_{r_{i+1}}}Q_i(|F_\epsilon|)+C\,Q_i(\mu)\,|B_{r_{i+1}}|+\frac{C}{Q_i(r_{i+1})}\,Q_{i+1}(\mu)\,|B_{r_{i+2}}|\,, 
\end{eqnarray}
where we apply a modified version of Sobolev inequality and \eqref{young1}\,. Iterating \eqref{equ19} from $i=0$ to $i=h-1$, one finds that
\begin{equation}
\int_{B_{\rho}}Q(|\nabla u_\epsilon|)\leq\tilde{C}_h\int_{B_{\lambda^h\rho}}Q_h(|u_\epsilon|)+\bar{C}_h\int_{B_{\lambda^h\rho}}Q(|F_\epsilon|)+\bar{C}_h\,Q(\mu)\,|B_{\lambda^h\rho}|.
\end{equation}
By virtue of the strong convergence of $u_\epsilon$ to $u$ in $W^{1,P}_{\text{loc}}(\Omega)$, we take $\epsilon\rightarrow0$ and obtain
\begin{eqnarray*}
\int_{B_{\rho}}Q(|\nabla u|)\leq\tilde{C}_h\int_{B_{\lambda^h\rho}}Q(|u|)+\bar{C}_h\int_{B_{\lambda^h\rho}}Q(|F|)+\bar{C}_h\,Q(\mu)\,|B_{\lambda^h\rho}|.
\end{eqnarray*}
Thus we get
\begin{eqnarray*}
\fint_BQ(|\nabla u|)\leq C\left(1+\fint_{\lambda B}Q(|u|)+\fint_{\lambda B}Q(|F|)\right),
\end{eqnarray*}
which completes the proof of Proposition \ref{higher int}.
\end{proof}

\section{Proofs of Theorem \ref{theorem1.3} and Theorem \ref{MainTheorem}}
In this section we present the proofs of Theorem \ref{theorem1.3} and \ref{MainTheorem}, respectively. Inspired by \cite{2}, the following lemma will be helpful. 

\begin{lemma}
Let $A:\Omega\times\mathbb{R}^n\rightarrow\mathbb{R}^n$ be a Carath\'eodory vector field such that {\rm{(A1)-(A3)}} and \eqref{1.4} hold. Then $A$ is locally uniformly in VMO, that is,
\begin{equation}
\lim_{R\rightarrow0}\,\,\sup_{r(B)<R}\,\,\sup_{c(B)\in K}\fint_B V(x,B)\operatorname{d}\!x=0,
\end{equation}
where $V(x,B)$ is defined in \eqref{Vx-BR}, $K\subset\Omega$, $c(B)$ and $r(B)$ denote the center and the radius of the ball $B$, respectively. 
\end{lemma}
\begin{proof}
From the definition \eqref{Vx-BR}, one has
\begin{eqnarray*}
\fint_{B}V(x,B)\operatorname{d}\!x&=&\fint_{B}\sup_{\xi\in\mathbb{R}^{n}}\frac{|A(x,\xi)-A_{B}(\xi)|}{\frac{P\left[\left(\mu^2+|\xi|^2\right)^{\frac{1}{2}}\right]}{\left(\mu^2+|\xi|^2\right)^{\frac{1}{2}}}}\operatorname{d}\!x\\
&\leq&\fint_{B}\sup_{\xi\in\mathbb{R}^{n}}\fint_{B}\frac{|A(x,\xi)-A(y,\xi)|}{\frac{P\left[\left(\mu^2+|\xi|^2\right)^{\frac{1}{2}}\right]}{\left(\mu^2+|\xi|^2\right)^{\frac{1}{2}}}}\operatorname{d}\!y\operatorname{d}\!x\\
&\leq&\fint_{B}\sup_{\xi\in\mathbb{R}^{n}}\fint_{B}(g(x)+g(y))|x-y|^\alpha\operatorname{d}\!y\operatorname{d}\!x\\
&\leq&\left(\fint_{B}\fint_{B}(g(x)+g(y))^{\frac{n}{\alpha}}\operatorname{d}\!y\operatorname{d}\!x\right)^{\frac{\alpha}{n}}\left(\fint_{B}\fint_{B}|x-y|^{\frac{n\alpha}{n-\alpha}}\operatorname{d}\!y\operatorname{d}\!x\right)^{\frac{n-\alpha}{n}}\\
&\leq&\left(\frac{1}{|B|}\fint_{B}g^{\frac{n}{\alpha}}\right)^{\frac{\alpha}{n}}\cdot C(\alpha, n)\,|B|^{\frac{\alpha}{n}}=C(n,\alpha)\left(\int_{B}g^{\frac{n}{\alpha}}\right)^{\frac{\alpha}{n}},
\end{eqnarray*}
and we complete the proof.
\end{proof}
We are in a position to present the proof. 
\begin{proof}[Proof of Theorem \ref{theorem1.3}]
We divide the proof into four steps. \\
\textbf{Step 1.} Assume that $B_{3R}\subset\Omega$. We construct a test function $\varphi=\Delta_{-h}(\eta^2\Delta_{h}u)$ to \eqref{1.5}, where $\eta\in C^\infty_0(\Omega)$ is a cut-off function satisfying that 
\begin{equation*}
0\leq\eta(x)\leq1,\ \ \eta(x)\equiv 1\ \text{for}\ x\in B_{\frac{R}{2}},\ \ \eta(x)\equiv 0\ \text{for}\ x\in\Omega\,\backslash\,B_R,\ \text{and}\ |\nabla\eta|\leq\frac{C}{R}. 
\end{equation*}
It follows that 
\begin{eqnarray}\label{4.14}
I_1&=&\int_{B_{2R}}[A(x+h,\nabla u(x+h))-A(x+h,\nabla u)]\cdot\eta^2\,\Delta_{h} \nabla u\operatorname{d}\!x\nonumber\\
&=&\int_{B_{2R}}[A(x,\nabla u)-A(x+h,\nabla u)]\cdot\eta^2\,\Delta_{h} \nabla u\operatorname{d}\!x\nonumber\\
&&+\,\,\int_{B_{2R}}[A(x+h,\nabla u)-A(x+h,\nabla u(x+h))]\cdot2\eta\,\nabla\eta\,\Delta_{h}u\operatorname{d}\!x\nonumber\\
&&+\,\,\int_{B_{2R}}[A(x,\nabla u)-A(x+h,\nabla u)]\cdot2\eta\,\nabla\eta\,\Delta_{h}u\operatorname{d}\!x\nonumber\\[0.1cm]
&=&I_{2}+I_{3}+I_{4}.
\end{eqnarray}
\textbf{Step 2.} We estimate each $I_{i}$ ($1\leq i\leq 4$) appeared in \eqref{4.14}. Using (A1), $I_1$ is estimated as
\begin{equation}
\begin{split}
\begin{aligned}
I_{1}\geq\upsilon\int_{B_{2R}}\frac{P\left[\left(\mu^2+|\nabla u(x+h)|^{2}+|\nabla u|^2\right)^{\frac{1}{2}}\right]}{\mu^2+|\nabla u(x+h)|^{2}+|\nabla u|^{2}}\,|\Delta_{h}\nabla u|^{2}\,\eta^2\operatorname{d}\!x. \nonumber\\
\end{aligned}
\end{split}
\end{equation}
For $I_{2}$, by virtue of the hypothesis \eqref{1.4} and \eqref{young2}, we acquire that
\begin{eqnarray*}
I_{2}&\leq&\int_{B_{2R}}|h|^{\alpha}(g(x)+g(x+h))\frac{P\left[\left(\mu^2+|\nabla u|^2\right)^{\frac{1}{2}}\right]}{\left(\mu^2+|\nabla u|^2\right)^{\frac{1}{2}}}\,\eta^{2}\,|\Delta_{h} \nabla u|\operatorname{d}\!x\nonumber\\
&\leq&\epsilon\int_{B_{2R}}\frac{P\left[\left(\mu^2+|\nabla u|^2\right)^{\frac{1}{2}}\right]}{\mu^2+|\nabla u|^2}\,|\Delta_{h} \nabla u|^{2}\,\eta^{2}\operatorname{d}\!x\nonumber\\
&&+\,\,C\,|h|^{2\alpha}\int_{B_{2R}}\left(g(x)+g(x+h)\right)^{2}\,P\left[\left(\mu^2+|\nabla u|^2\right)^{\frac{1}{2}}\right]\operatorname{d}\!x.
\end{eqnarray*}
By the assumption (P1), we have
\begin{eqnarray*}
&&\epsilon\int_{B_{2R}}\frac{P\left[\left(\mu^2+|\nabla u|^2\right)^{\frac{1}{2}}\right]}{\mu^2+|\nabla u|^2}\,|\Delta_{h} \nabla u|^{2}\,\eta^{2}\operatorname{d}\!x\\
&\leq&\epsilon\int_{B_{2R}}\frac{P\left[\left(\mu^2+|\nabla u(x+h)|^{2}+|\nabla u|^2\right)^{\frac{1}{2}}\right]}{\mu^2+|\nabla u(x+h)|^{2}+|\nabla u|^{2}}\,|\Delta_{h} \nabla u|^{2}\,\eta^{2}\operatorname{d}\!x,
\end{eqnarray*}
where $\epsilon$ will be chosen later. According to (A2) and \eqref{young2}, we get an estimate for $I_{3}$ as 
\begin{eqnarray*}
I_3&\leq&C\int_{B_{2R}}\frac{P\left[\left(\mu^2+|\nabla u|^2+|\nabla u(x+h)|^2\right)^{\frac{1}{2}}\right]}{\mu^2+|\nabla u|^2+|\nabla u(x+h)|^2}\,|\Delta_{h}\nabla u|\,|\eta|\,|\nabla\eta|\,|\Delta_{h} u|\operatorname{d}\!x\\
&\leq&\epsilon\int_{B_{2R}}\frac{P\left[\left(\mu^2+|\nabla u|^2+|\nabla u(x+h)|^2\right)^{\frac{1}{2}}\right]}{\mu^2+|\nabla u|^2+|\nabla u(x+h)|^2}\,|\Delta_{h}\nabla u|^{2}\,\eta^2\operatorname{d}\!x\\
&&+\,\,C\int_{B_{2R}}\frac{P\left[\left(\mu^2+|\nabla u|^2+|\nabla u(x+h)|^2\right)^{\frac{1}{2}}\right]}{\mu^2+|\nabla u|^2+|\nabla u(x+h)|^2}\,|\nabla\eta|^2\,|\Delta_{h}u|^2\operatorname{d}\!x.
\end{eqnarray*}
 In view of the definition of $N$-function and Lagrange Mean Value Theorem, we obtain
\begin{eqnarray}\label{estimate-I3}
&&C\int_{B_{2R}}\frac{P\left[\left(\mu^2+|\nabla u|^2+|\nabla u(x+h)|^2\right)^{\frac{1}{2}}\right]}{\mu^2+|\nabla u|^2+|\nabla u(x+h)|^2}\,|\nabla\eta|^2\,|\Delta_{h} u|^2\operatorname{d}\!x\\
&\leq&C\,|h|^2\int_{B_{2R+|h|}}\frac{P\left[\left(\mu^2+2|\nabla u|^2\right)^{\frac{1}{2}}\right]}{\mu^2+2|\nabla u|^2}\,|\nabla u|^2\operatorname{d}\!x \nonumber \\
&\leq&C\,|h|^2\int_{B_{2R+|h|}}P(\mu+|\nabla u|)\operatorname{d}\!x. \nonumber 
\end{eqnarray}
To estimate $I_4$, the hypothesis \eqref{1.4} and \eqref{young2} give us that 
\begin{eqnarray*}
I_4&\leq&C\int_{B_{2R}}|h|^{\alpha}\left(g(x)+g(x+h)\right)\frac{P\left[\left(\mu^2+|\nabla u|^2\right)^{\frac{1}{2}}\right]}{(\mu^2+|\nabla u|^2)^{\frac{1}{2}}}\,|\eta|\,|\nabla\eta|\,|\Delta_{h} u|\operatorname{d}\!x\\
&\leq&\epsilon\int_{B_{2R}}\frac{P\left[\left(\mu^2+|\nabla u|^2\right)^{\frac{1}{2}}\right]}{\mu^2+|\nabla u|^2}\,\eta^2\,|\Delta_{h} u|^2\operatorname{d}\!x\\
&&+\,\,C\,|h|^{2\alpha}\int_{B_{2R}}(g(x)+g(x+h))^2\,P\left[\left(\mu^2+|\nabla u|^2\right)^{\frac{1}{2}}\right]\operatorname{d}\!x.
\end{eqnarray*}
As in the estimate of $I_3$, we get 
\begin{eqnarray*}
&&\epsilon\int_{B_{2R}}\frac{P\left[\left(\mu^2+|\nabla u|^2\right)^{\frac{1}{2}}\right]}{\mu^2+|\nabla u|^2}\,\eta^2\,|\Delta_{h} u|^2\operatorname{d}\!x\leq C\,|h|^2\int_{B_{2R+|h|}}P(\mu+|\nabla u|)\operatorname{d}\!x.
\end{eqnarray*}
Collecting the estimates of $I_i$, we evidently get if $0<\mu<1$, then 
\begin{eqnarray}
&&\left(\upsilon-2\epsilon\right)\int_{B_{2R}}\frac{P\left[\left(\mu^2+|\nabla u(x+h)|^{2}+|\nabla u|^2\right)^{\frac{1}{2}}\right]}{\mu^2+|\nabla u(x+h)|^{2}+|\nabla u|^{2}}|\Delta_{h} \nabla u|^{2}\,\eta^{2}\operatorname{d}\!x\nonumber\\
&\!\!\leq\!\!&C\,|h|^{2\alpha}\int_{B_{2R}}(g(x)+g(x+h))^2\,P\left[\left(\mu^2+|\nabla u|^2\right)^\frac{1}{2}\right]\operatorname{d}\!x\nonumber\\
&&+\,\,C\,|h|^2\int_{B_{2R+|h|}}P(\mu+|\nabla u|)\operatorname{d}\!x.
\end{eqnarray}
Choosing $\epsilon$ small enough, we obtain
\begin{eqnarray}
&&\int_{B_{2R}}\frac{P\left[\left(\mu^2+|\nabla u(x+h)|^{2}+|\nabla u|^2\right)^{\frac{1}{2}}\right]}{\mu^2+|\nabla u(x+h)|^{2}+|\nabla u|^{2}}|\Delta_{h} \nabla u|^{2}\,\eta^{2}\operatorname{d}\!x\nonumber\\
&\!\!\leq\!\!&C\,|h|^{2\alpha}\int_{B_{2R}}(g(x)+g(x+h))^2\,P\left[\left(\mu^2+|\nabla u|^2\right)^\frac{1}{2}\right]\operatorname{d}\!x\nonumber\\
&&+\,\,C\,|h|^2\int_{B_{2R+|h|}}P(\mu+|\nabla u|)\operatorname{d}\!x.
\end{eqnarray}
\textbf{Step 3.} 
By the assumption (P2), one gets 
\begin{equation}\label{equ22}
|\Delta_{h} V_P|^2\leq C\,\frac{P\left[\left(\mu^2+|\nabla u(x+h)|^2+|\nabla u|^2\right)^{\frac{1}{2}}\right]}{\mu^2+|\nabla u(x+h)|^2+|\nabla u|^2}\,|\Delta_{h}\nabla u|^2. 
\end{equation}
We integrate both sides of \eqref{equ22}, use the properties of $\eta$, and acquire
\begin{eqnarray}\label{4.7}
\int_{B_{\frac{R}{2}}}|\Delta_{h} V_P|^2\operatorname{d}\!x&\leq&C\int_{B_{\frac{R}{2}}}\frac{P\left[\left(\mu^2+|\nabla u(x+h)|^2+|\nabla u|^2\right)^{\frac{1}{2}}\right]}{\mu^2+|\nabla u(x+h)|^2+|\nabla u|^2}|\Delta_{h} \nabla u|^2\,\eta^2\operatorname{d}\!x\nonumber\\
&\leq&C\,|h|^{2\alpha}\int_{B_{2R}}(g(x)+g(x+h))^2\,P\left[\left(\mu^2+|\nabla u|^2\right)^\frac{1}{2}\right]\operatorname{d}\!x\nonumber\\
&&+\,\,C\,|h|^2\int_{B_{2R+|h|}}P(\mu+|\nabla u|)\operatorname{d}\!x.
\end{eqnarray}
Dividing both sides of \eqref{4.7} by $|h|^{2\alpha}$, it follows that 
\begin{eqnarray}
\int_{B_{\frac{R}{2}}}\left|\frac{\Delta_{h}V_P}{|h|^\alpha}\right|^2\operatorname{d}\!x&\leq&C\int_{B_{2R}}(g(x)+g(x+h))^2\,P\left[\left(\mu^2+|\nabla u|^2\right)^\frac{1}{2}\right]\operatorname{d}\!x\nonumber\\
&&+\,\,C\,|h|^{2-2\alpha}\int_{B_{2R+|h|}}P(\mu+|\nabla u|)\operatorname{d}\!x\nonumber\\
&=:&J_1+J_2.
\end{eqnarray}
\textbf{Step 4.} We shall prove that $J_i$ is bounded for each $i$. By Proposition \ref{higher int}, we get $P(|\nabla u|)\in L^{t}(\Omega)$ with $t>1$. In particular, $P(|\nabla u|)\in L^{\frac{n}{n-2\alpha}}(\Omega)$. By choosing $0<|h|<\delta<R$ and \eqref{1.4}, we obtain
\begin{eqnarray*}
J_1&\leq&C\left(\int_{B_{2R}}(g(x)+g(x+h))^{\frac{n}{\alpha}}\operatorname{d}\!x\right)^{\frac{2\alpha}{n}}\left(\int_{B_{2R}}P\left[\left(\mu^2+|\nabla u|^2\right)^{\frac{1}{2}}\right]^{\frac{n}{n-2\alpha}}\operatorname{d}\!x\right)^{\frac{n-2\alpha}{n}}\\
&\leq&C\left(\int_{B_{2R+|h|}}g(x)^{\frac{n}{\alpha}}\operatorname{d}\!x\right)^{\frac{2\alpha}{n}}\left(\int_{B_{2R}}P\left[\left(\mu^2+|\nabla u|^2\right)^{\frac{1}{2}}\right]^{\frac{n}{n-2\alpha}}\operatorname{d}\!x\right)^{\frac{n-2\alpha}{n}}<+\infty.
\end{eqnarray*}
Because $u\in W^{1,P}_{\text{loc}}(\Omega)$, we get $J_2<+\infty$. It follows that $\sup\limits_{|h|<\delta}\int_{B_{\frac{R}{2}}}\left|\frac{\Delta_hV_P}{|h|^\alpha}\right|^2\operatorname{d}\!x<+\infty$ with $\delta<R$, that is,  $V_P(\nabla u)\in B^\alpha_{2,\infty}(\Omega)$ locally.
\end{proof}
We also need the following lemma. 
\begin{lemma}
Let $A:\Omega\times\mathbb{R}^n\rightarrow\mathbb{R}^n$ be a Carath\'eodory vector field such that {\rm{(A1)-(A4)}} hold. Then $A$ is locally uniformly in VMO, that is,
\begin{equation}
\lim_{R\rightarrow0}\,\,\sup_{r(B)<R}\,\,\sup_{c(B)\in K}\fint_BV(x,B)\operatorname{d}\!x=0,
\end{equation}
where $V(x,B)$ is given in \eqref{Vx-BR}, $K\subset\Omega$, $c(B)$ and $r(B)$ denote the center and the radius of the ball $B$, respectively. 
\end{lemma}
\begin{proof}
Given a point $x\in\Omega$, let us write $A_k(x)=\{y\in\Omega:2^{-k}\leq|x-y|<2^{-k+1}\}$. We have
\begin{eqnarray*}
\fint_{B}V(x,B)\operatorname{d}\!x&\leq&\fint_{B}\sup_{\xi\in\mathbb{R}^{n}}\fint_{B}\frac{|A(x,\xi)-A(y,\xi)|}{\frac{P\left[\left(\mu^2+|\xi|^2\right)^{\frac{1}{2}}\right]}{\left(\mu^2+|\xi|^2\right)^{\frac{1}{2}}}}\operatorname{d}\!y\operatorname{d}\!x\\
&=&\fint_{B}\sup_{\xi\in\mathbb{R}^{n}}\frac{1}{|B|}\sum_{k}\int_{B\,\cap\,A_k(x)}\frac{|A(x,\xi)-A(y,\xi)|}{\frac{P\left[\left(\mu^2+|\xi|^2\right)^{\frac{1}{2}}\right]}{\left(\mu^2+|\xi|^2\right)^{\frac{1}{2}}}}\operatorname{d}\!y\operatorname{d}\!x\\
&\leq&\frac{1}{|B|^2}\sum_{k}\int_B\int_{B\,\cap\,A_k(x)}|x-y|^\alpha(g_k(x)+g_k(y))\operatorname{d}\!y\operatorname{d}\!x\\
&\leq&\left(\frac{1}{|B|^2}\sum_k\int_B\int_{B\,\cap\,A_k(x)}|x-y|^{\frac{n\alpha}{n-\alpha}}\operatorname{d}\!y\operatorname{d}\!x\right)^{\frac{n-\alpha}{n}}\\
&&\ \ \cdot\left(\frac{1}{|B|^2}\sum_k\int_{B}\int_{B\,\cap\,A_k(x)}(g_k(x)+g_k(y))^{\frac{n}{\alpha}}\operatorname{d}\!y\operatorname{d}\!x\right)^{\frac{\alpha}{n}}\\
&\leq&C(n,\alpha)|B|^{\frac{\alpha}{n}}\left(\frac{1}{|B|^2}\sum_k\int_{B}\int_{B\,\cap\,A_k(x)}(g_k(x)+g_k(y))^{\frac{n}{\alpha}}\operatorname{d}\!y\operatorname{d}\!x\right)^{\frac{\alpha}{n}}.
\end{eqnarray*}
By H\"older inequality, one obtains that 
\begin{eqnarray*}
&&\left(\frac{1}{|B|^2}\sum_k\int_{B}\int_{B\,\cap\,A_k(x)}(g_k(x)+g_k(y))^{\frac{n}{\alpha}}\operatorname{d}\!y\operatorname{d}\!x\right)^{\frac{\alpha}{n}}\\
&\leq&C\left(\frac{1}{|B|^2}\sum_k|B\,\cap\,A_k(x)|\int_Bg_k(x)^{\frac{n}{\alpha}}\operatorname{d}\!x\right)^{\frac{\alpha}{n}}\\
&\leq&\frac{C}{|B|^{\frac{2}{q}}}\left(\sum_k\|g_k\|^q_{L^{\frac{n}{\alpha}}(B)}\right)^{\frac{1}{q}}\frac{1}{|B|^{2\left(\frac{\alpha}{n}-\frac{1}{q}\right)}}\left(\sum_k|B\,\cap\,A_k(x)|^{\frac{\alpha q}{\alpha q-n}}\right)^{\frac{\alpha}{n}\frac{\alpha q-n}{\alpha q}}\\
&\leq&C(n,\alpha, q)\,|B|^{-\frac{\alpha}{n}}\left(\sum_k\|g_k\|^q_{L^{\frac{n}{\alpha}}(B)}\right)^{\frac{1}{q}}.
\end{eqnarray*}
We fix $r>0$ small enough and observe that $x\rightarrow\|g_k\|_{l^q(L^{\frac{n}{\alpha}}(B_{r}(x)))}$ is continuous on the set $\{x\in\Omega:{\operatorname{dist}}(x,\partial\Omega)>r\}$. As a consequence, for $r>0$ small enough, there is a point $x_r\in K$ such that
\begin{eqnarray*}
&&\sup_{x\in K}\|g_k\|_{l^q({L^{\frac{n}{\alpha}}(B_{r}(x))})}=\|g_k\|_{l^q({L^{\frac{n}{\alpha}}(B_{r}(x))})}.
\end{eqnarray*}
We obtain that
\begin{eqnarray*}
&&\lim_{r\rightarrow0}\|g_k\|_{l^q({L^{\frac{n}{\alpha}}(B_{r}(x))})}=\left(\sum_k\lim_{r\rightarrow0}\left(\int_{B_{r}(x_r)}g_k^{\frac{n}{\alpha}}\right)^{\frac{q\alpha}{n}}\right)^{\frac{1}{q}}.
\end{eqnarray*}
Each of the limits on the right hand side equals to $0$, and so it completes the proof.
\end{proof}
With the help of preceding lemma, we have the following. 
\begin{proof}[Proof of Theorem \ref{MainTheorem}]
We divide the proof in four steps. \\
\textbf{Step 1.} We assume that $B_{3R+1}\subset\Omega$, and choose a test function $\varphi=\Delta_{-h}(\eta^2\Delta_{h}u)$ to \eqref{equ}, where $\eta\in C^\infty_0(\Omega)$ is a cut-off function satisfying that 
\begin{equation*}
0\leq\eta(x)\leq1,\ \ \eta(x)\equiv 1\ \text{for}\ x\in B_{\frac{R}{2}},\ \ \eta(x)\equiv 0\ \text{for}\ x\in\Omega\,\backslash\,B_R,\ \text{and}\ |\nabla\eta|\leq\frac{C}{R}\,. 
\end{equation*}
According to the definition of weak solution and choice of test function, we have 
\begin{eqnarray}\label{equ30}
I_1&=&\int_{B_{2R}}[A(x+h,\nabla u(x+h))-A(x+h,\nabla u)]\cdot\eta^2\,\Delta_{h} \nabla u\operatorname{d}\!x\nonumber\\
&=&\int_{B_{2R}}[A(x,\nabla u)-A(x+h,\nabla u)]\cdot\eta^2\,\Delta_{h} \nabla u\operatorname{d}\!x\nonumber\\
&&+\,\,\int_{B_{2R}}[A(x+h,\nabla u)-A(x+h,\nabla u(x+h))]\cdot2\eta\,\nabla\eta\,\Delta_{h}u\operatorname{d}\!x\nonumber\\
&&+\,\,\int_{B_{2R}}[A(x,\nabla u)-A(x+h,\nabla u)]\cdot2\eta\,\nabla\eta\,\Delta_{h}u\operatorname{d}\!x\nonumber\\
&&+\,\,\int_{B_{2R}}\Delta_{h}\left[\frac{P(|F|)}{|F|^2}F\right]\cdot2\eta\,\nabla\eta\,\Delta_{h}u\operatorname{d}\!x+\int_{B_{2R}}\Delta_{h}\left[\frac{P(|F|)}{|F|^2}F\right]\cdot\eta^2\,\Delta_{h} \nabla u\operatorname{d}\!x\nonumber\\
&=&I_{2}+I_{3}+I_{4}+I_{5}+I_{6}.
\end{eqnarray}
\textbf{Step 2.} One can estimate the terms $I_1$ to $I_4$ as in the proof of Theorem \ref{theorem1.3}. We should mention that, in order to  evaluate $I_{2}$ to $I_4$, we apply the assumption (A4) instead of \eqref{1.4}, and substitute $g$ by $g_{k}$. Thus it remains to estimate $I_5$ and $I_6$. For $I_5$, we apply \eqref{young2} to get
\begin{eqnarray*}
I_5&\leq&C\int_{B_{2R}}\left|\Delta_{h}\left[\frac{P(|F|)}{|F|^2}F\right]\right|\,|\Delta_{h}u|\,|\eta|\operatorname{d}\!x\\
&\leq&C\int_{B_{2R}}\left|\Delta_{h}\left[\frac{P(|F|)}{|F|^2}F\right]\right|^2\operatorname{d}\!x+C\int_{B_{2R}}|\Delta_{h}u|^2\,\eta^2\operatorname{d}\!x\\
&\leq&C\,|h|^{2\alpha}\int_{B_{2R}}\left|\frac{\Delta_{h}\left[\frac{P(|F|)}{|F|^2}F\right]}{|h|^{\alpha}}\right|^2\operatorname{d}\!x+C\int_{B_{2R}}|\Delta_{h}u|^2\,\eta^2\operatorname{d}\!x.
\end{eqnarray*}
By the assumption (P1) and the Lagrange mean value theorem, we acquire
\begin{eqnarray*}
C\int_{B_{2R}}|\Delta_{h}u|^2\,\eta^2\operatorname{d}\!x&\leq&C\int_{B_{2R}}\frac{P(\mu)}{\mu^2}\,|\Delta_{h}u|^2\,\eta^2\operatorname{d}\!x\\
&\leq&C\,|h|^2\int_{B_{2R+|h|}}\frac{P\left[\left(\mu^2+|\nabla u|^2\right)^{\frac{1}{2}}\right]}{\mu^2+|\nabla u|^2}\,|\nabla u|^2\operatorname{d}\!x\\
&\leq&C\,|h|^2\int_{B_{2R+|h|}}P(\mu+|\nabla u|)\operatorname{d}\!x.
\end{eqnarray*}
As in the case of estimating $I_5$, it is apparent that
\begin{eqnarray*}
I_6&\leq&\int_{B_{2R}}\left|\Delta_{h}\left[\frac{P(|F|)}{|F|^2}F\right]\right|\,|\Delta_{h}\nabla u|\,\eta^2\operatorname{d}\!x\\
&\leq&C\int_{B_{2R}}\left|\Delta_{h}\left[\frac{P(|F|)}{|F|^2}F\right]\right|^2\operatorname{d}\!x+\epsilon\int_{B_{2R}}|\Delta_{h}\nabla u|^2\,\eta^2\operatorname{d}\!x\\
&=&C\,|h|^{2\alpha}\int_{B_{2R}}\left|\frac{\Delta_{h}\left[\frac{P(|F|)}{|F|^2}F\right]}{|h|^{\alpha}}\right|^2\operatorname{d}\!x+\epsilon\int_{B_{2R}}|\Delta_{h}\nabla u|^2\,\eta^2\operatorname{d}\!x.
\end{eqnarray*}
Using the fact of $0<\mu<1$ and the assumption (P1), we acquire
\begin{eqnarray*}
&&\epsilon\int_{B_{2R}}|\Delta_{h}\nabla u|^2\,\eta^2\operatorname{d}\!x\\
&\leq&\frac{\epsilon\mu^2}{P(\mu)}\int_{B_{2R}}\frac{P(\mu)}{\mu^2}|\Delta_{h}\nabla u|^2\,\eta^2\operatorname{d}\!x\\
&\leq&\frac{\epsilon\mu^2}{P(\mu)}\int_{B_{2R}}\frac{P\left[\left(\mu^2+|\nabla u(x+h)|^{2}+|\nabla u|^2\right)^{\frac{1}{2}}\right]}{\mu^2+|\nabla u(x+h)|^{2}+|\nabla u|^{2}}|\Delta_{h} \nabla u|^{2}\,\eta^{2}\operatorname{d}\!x.
\end{eqnarray*}
Collecting the estimates of $I_i$, we evidently get if $0<\mu<1$, then there holds 
\begin{eqnarray}\label{4.4}
&&\left(\upsilon-2\epsilon-\frac{\epsilon\mu^2}{P(\mu)}\right)\int_{B_{2R}}\frac{P\left[\left(\mu^2+|\nabla u(x+h)|^{2}+|\nabla u|^2\right)^{\frac{1}{2}}\right]}{\mu^2+|\nabla u(x+h)|^{2}+|\nabla u|^{2}}|\Delta_{h} \nabla u|^{2}\,\eta^{2}\operatorname{d}\!x\nonumber\\
&\!\!\leq\!\!&C\,|h|^{2\alpha}\int_{B_{2R}}(g_k(x)+g_k(x+h))^2\,P\left[\left(\mu^2+|\nabla u|^2\right)^\frac{1}{2}\right]\operatorname{d}\!x\nonumber\\
&&+\,\,C\,|h|^2\int_{B_{2R+|h|}}P(\mu+|\nabla u|)\operatorname{d}\!x+C\,|h|^{2\alpha}\int_{B_{2R}}\left|\frac{\Delta_{h}\left[\frac{P(|F|)}{|F|^2}F\right]}{|h|^{\alpha}}\right|^2\operatorname{d}\!x.
\end{eqnarray}
By choosing $\epsilon$ small enough, we obtain that
\begin{eqnarray}
&&\int_{B_{2R}}\frac{P\left[\left(\mu^2+|\nabla u(x+h)|^{2}+|\nabla u|^2\right)^{\frac{1}{2}}\right]}{\mu^2+|\nabla u(x+h)|^{2}+|\nabla u|^{2}}|\Delta_{h} \nabla u|^{2}\,\eta^{2}\operatorname{d}\!x\nonumber\\
&\leq&C\,|h|^{2\alpha}\int_{B_{2R}}(g_k(x)+g_k(x+h))^2\,P\left[\left(\mu^2+|\nabla u|^2\right)^\frac{1}{2}\right]\operatorname{d}\!x\nonumber\\
&&+\,\,C\,|h|^2\int_{B_{2R+|h|}}P(\mu+|\nabla u|)\operatorname{d}\!x+C\,|h|^{2\alpha}\int_{B_{2R}}\left|\frac{\Delta_{h}\left[\frac{P(|F|)}{|F|^2}F\right]}{|h|^{\alpha}}\right|^2\operatorname{d}\!x.
\end{eqnarray}
\textbf{Step 3.} 
By the auxiliary function \eqref{Vp-def} and the assumption (P2), we obtain that
\begin{equation}\label{4.12}
|\Delta_{h} V_P|^2\leq C\frac{P\left[\left(\mu^2+|\nabla u(x+h)|^2+|\nabla u|^2\right)^{\frac{1}{2}}\right]}{\mu^2+|\nabla u(x+h)|^2+|\nabla u|^2}\,|\Delta_{h} \nabla u|^2. \\
\end{equation}
Integrating both sides of \eqref{4.12}, we acquire
\begin{eqnarray}\label{equ12}
\int_{B_{\frac{R}{2}}}|\Delta_{h} V_P|^2\operatorname{d}\!x&\leq&C\int_{B_{\frac{R}{2}}}\frac{P\left[\left(\mu^2+|\nabla u(x+h)|^2+|\nabla u|^2\right)^{\frac{1}{2}}\right]}{\mu^2+|\nabla u(x+h)|^2+|\nabla u|^2}|\Delta_{h} \nabla u|^2\,\eta^2\operatorname{d}\!x\nonumber\\
&\leq&C\,|h|^2\int_{B_{2R+|h|}}P(\mu+|\nabla u|)\operatorname{d}\!x+C\,|h|^{2\alpha}\int_{B_{2R}}\left|\frac{\Delta_{h}\left[\frac{P(|F|)}{|F|^2}F\right]}{|h|^{\alpha}}\right|^2\operatorname{d}\!x\nonumber\\
&&+\,\,C\,|h|^{2\alpha}\int_{B_{2R}}(g_k(x)+g_k(x+h))^2\,P\left[\left(\mu^2+|\nabla u|^2\right)^\frac{1}{2}\right]\operatorname{d}\!x.
\end{eqnarray}
Dividing both sides of \eqref{equ12} by $|h|^{2\alpha}$ and applying the properties of $\eta$, one derives that 
\begin{eqnarray}\label{equ13}
\int_{B_{\frac{R}{2}}}\left|\frac{\Delta_{h}V_P}{|h|^{\alpha}}\right|^2\operatorname{d}\!x&\leq&C\,|h|^{2-2\alpha}\int_{B_{2R+|h|}}P(\mu+|\nabla u|)\operatorname{d}\!x+C\int_{B_{2R}}\left|\frac{\Delta_{h}\left[\frac{P(|F|)}{|F|^2}F\right]}{|h|^{\alpha}}\right|^2\operatorname{d}\!x\nonumber\\ 
&&+\,\,C\int_{B_{2R}}(g_k(x)+g_k(x+h))^2\,P\left[\left(\mu^2+|\nabla u|^2\right)^\frac{1}{2}\right]\operatorname{d}\!x.
\end{eqnarray}
We raise both sides of \eqref{equ13} to the power of $\frac{1}{2}$ and get
\begin{eqnarray}\label{equ23}
\left(\int_{B_{\frac{R}{2}}}\left|\frac{\Delta_{h}V_P}{|h|^{\alpha}}\right|^2\operatorname{d}\!x\right)^{\frac{1}{2}}&\!\!\leq\!\!&C\left(\int_{B_{2R}}(g_k(x)+g_k(x+h))^2\,P\left[\left(\mu^2+|\nabla u|^2\right)^\frac{1}{2}\right]\operatorname{d}\!x\right)^{\frac{1}{2}}\nonumber\\
&&+\,\,C\,|h|^{1-\alpha}\left(\int_{B_{2R+|h|}}P(\mu+|\nabla u|)\operatorname{d}\!x\right)^{\frac{1}{2}}\nonumber\\
&&+\,\,C\left(\int_{B_{2R}}\left|\frac{\Delta_{h}\left[\frac{P(|F|)}{|F|^2}F\right]}{|h|^{\alpha}}\right|^2\operatorname{d}\!x\right)^{\frac{1}{2}}.
\end{eqnarray}
Restricting to $B_\delta$ with $0<|h|<\delta$ and taking the $L^q$ norm with respect to the measure $\frac{\operatorname{d}\!h}{|h|^n}$, it follows that 
\begin{eqnarray}
&&\left(\int_{B_{\delta}}\left(\int_{B_{\frac{R}{2}}}\left|\frac{\Delta_{h}V_P}{|h|^{\alpha}}\right|^2\operatorname{d}\!x\right)^{\frac{q}{2}}\frac{\operatorname{d}\!h}{|h|^n}\right)^{\frac{1}{q}} \nonumber \\
&\leq& C\left(\int_{B_{\delta}}\left(\int_{B_{2R}}(g_k(x)+g_k(x+h))^2\,P\left[\left(\mu^2+|\nabla u|^2\right)^{\frac{1}{2}}\right]\operatorname{d}\!x\right)^{\frac{q}{2}}\frac{\operatorname{d}\!h}{|h|^n}\right)^{\frac{1}{q}} \nonumber\\
&&+\,\,C\left(\int_{B_{\delta}}|h|^{(1-\alpha)q}\left(\int_{B_{2R+|h|}}P(\mu+|\nabla u|)\operatorname{d}\!x\right)^{\frac{q}{2}}\frac{\operatorname{d}\!h}{|h|^n}\right)^{\frac{1}{q}} \nonumber\\
&&+\,\,C\left(\int_{B_{\delta}}\left(\int_{B_{2R}}\left|\frac{\Delta_{h}\left[\frac{P(|F|)}{|F|^2}F\right]}{|h|^{\alpha}}\right|^2\operatorname{d}\!x\right)^{\frac{q}{2}}\frac{\operatorname{d}\!h}{|h|^n}\right)^{\frac{1}{q}}\nonumber\\
&=:&J_1+J_2+J_3. 
\end{eqnarray}
\textbf{Step 4.} We shall show that each $J_i$ is bounded. Since $B^\alpha_{2,q}(\Omega)\subset L^{\frac{2n}{n-2\alpha}}(\Omega)$ with $1\leq q<\frac{2n}{n-2\alpha}$, one has $\frac{P(|F|)}{|F|^2}F\in L^{\frac{2n}{n-2\alpha}}(\Omega)$. By Proposition \ref{higher int}, we get $\frac{P(\mu+|\nabla u|)}{(\mu+|\nabla u|)^2}\nabla u\in L^{\frac{2n}{n-2\alpha}}(\Omega)$. This implies that $\nabla u\in L^{\frac{2n}{n-2\alpha}}(\Omega)$ and so $P(|\nabla u|)\in L^{\frac{n}{n-2\alpha}}(\Omega)$. \\
\indent To estimate $J_{1}$, we describe the $L^q$ norm in polar coordinates. Without loss of generality, we assume that $\delta=1$, hence $h\in B_{1}$ is equivalent to $h=r\xi$ for $0\leq r<1$ and $\xi$ in the unit sphere $\mathbb{S}^{n-1}$. Let $\operatorname{d}\!\sigma(\xi)$ be the surface measure on $\mathbb{S}^{n-1}$. By setting $r_k=\frac{1}{2^k}$, we bound $J_1$ by 
\begin{eqnarray}
J_1&=& \int^1_0\int_{\mathbb{S}^{n-1}}\left(\int_{B_{2R}}(g_k(x+r\xi)+g_k(x))^2\,P\left[\left(\mu^2+|\nabla u|^2\right)^{\frac{1}{2}}\right]\operatorname{d}\!x\right)^{\frac{q}{2}}\operatorname{d}\!\sigma(\xi)\frac{\operatorname{d}\!r}{r}\nonumber\\
&=& \sum^\infty_{k=0}\int^{r_k}_{r_{k+1}}\int_{\mathbb{S}^{n-1}}\left(\int_{B_{2R}}(g_k(x+r\xi)+g_k(x))^2\, P\left[\left(\mu^2+|\nabla u|^2\right)^{\frac{1}{2}}\right]\operatorname{d}\!x\right)^{\frac{q}{2}}\operatorname{d}\!\sigma(\xi)\frac{\operatorname{d}\!r}{r}\nonumber\\
&\leq& \sum^\infty_{k=0}\int^{r_k}_{r_{k+1}}\int_{\mathbb{S}^{n-1}}\left\|(\tau_{r\xi}\,g_k+g_k)\,\left(P\left[\left(\mu^2+|\nabla u|^2\right)^{\frac{1}{2}}\right]\right)^{\frac{1}{2}}\right\|^q_{L^2(B_{2R})}\operatorname{d}\!\sigma(\xi)\frac{\operatorname{d}\!r}{r}. \nonumber
\end{eqnarray}
Here $\tau_{r\xi}\,g_{k}(x)=g_{k}(x+r\xi)$. Since $P(|\nabla u|)\in L^{\frac{n}{n-2\alpha}}(\Omega)$ and $g_k\in L^{\frac{n}{\alpha}}(\Omega)$, we obtain that 
\begin{eqnarray*}
&&\left\|(\tau_{r\xi}\,g_k+g_k)\,\left(P\left[\left(\mu^2+|\nabla u|^2\right)^{\frac{1}{2}}\right]\right)^{\frac{1}{2}}\right\|_{L^2(B_{2R})}\nonumber\\
&=&\left(\int_{B_{2R}}(\tau_{r\xi}\,g_k+g_k)^2\,P\left[\left(\mu^2+|\nabla u|^2\right)^{\frac{1}{2}}\right]\operatorname{d}\!x\right)^{\frac{1}{2}}\nonumber\\
&\leq&\left(\left[\int_{B_{2R}}(\tau_{r\xi}\,g_k+g_k)^{2\,\frac{n}{2\alpha}}\operatorname{d}\!x\right]^{\frac{2\alpha}{n}}\left[\int_{B_{2R}}P\left[\left(\mu^2+|\nabla u|^2\right)^{\frac{1}{2}}\right]^{\frac{n}{n-2\alpha}}\operatorname{d}\!x\right]^{\frac{n-2\alpha}{n}}\right)^{\frac{1}{2}}\nonumber\\
&=&\|(\tau_{r\xi}\,g_k+g_k)\|_{L^{\frac{n}{\alpha}}(B_{2R})}\,\left\|P\left[\left(\mu^2+|\nabla u|^2\right)^{\frac{1}{2}}\right]\right\|^{\frac{1}{2}}_{L^{\frac{n}{n-2\alpha}}(B_{2R})}.
\end{eqnarray*}
On the other hand, noting that for each $\xi\in \mathbb{S}^{n-1}$ and $r_{k+1}\leq r\leq r_k$, there holds 
\begin{equation}
\|(\tau_{r\xi}\,g_k+g_k)\|_{L^{\frac{n}{\alpha}}(B_{2R})}\leq\|g_k\|_{L^{\frac{n}{\alpha}}((B_{2R})-r_{k}\xi)}+\|g_k\|_{L^{\frac{n}{\alpha}}(B_{2R})}\leq2\,\|g_k\|_{L^{\frac{n}{\alpha}}(\varrho B_R)},\nonumber
\end{equation}
where $\varrho=3+\frac{1}{R}$. Hence one gets 
\begin{equation}
J_1\,\leq\,C\,\left\|P\left[\left(\mu^2+|\nabla u|^2\right)^{\frac{1}{2}}\right]\right\|^{\frac{1}{2}}_{L^{\frac{n}{n-2\alpha}}(B_{2R})}\,\|\{g_k\}_k\|_{l^q\left(L^{\frac{n}{\alpha}}(\varrho B_R)\right)}<\infty. \nonumber
\end{equation}
Using the fact that $u\in W^{1,P}(\Omega)$, we deduce that 
\begin{eqnarray*}
J_2&\leq&C\left(\int_{B_\delta}|h|^{(1-\alpha)q-n}\operatorname{d}\!x\right)^{\frac{1}{q}}\left(\int_{B_{2R+|h|}}P(\mu+|\nabla u|)\operatorname{d}\!x\right)^{\frac{1}{2}}\\
&\leq&C\left(\int^\delta_0\rho^{(1-\alpha)q-1}\operatorname{d}\!\rho\right)^{\frac{1}{q}}\left(\int_{B_{2R+|h|}}P(\mu+|\nabla u|)\operatorname{d}\!x\right)^{\frac{1}{2}}<\infty.
\end{eqnarray*}
As $\frac{P(|F|)}{|F|^2}F\in B^\alpha_{2,q}(\Omega)$, it follows that 
\begin{equation}
J_3=C\,\left\|\frac{\Delta_{h}\left[\frac{P(|F|)}{|F|^2}F\right]}{|h|^\alpha}\right\|_{L^q\left(\frac{\operatorname{d}\!h}{|h|^n};L^2(B_{2R})\right)}<\infty. \nonumber\\
\end{equation}
Therefore, we complete the proof of Theorem \ref{MainTheorem}. 
\end{proof}

\section*{Acknowledgments}
The authors are supported by the National Natural Science Foundation of China (NNSF Grant No.~12001333) and Shandong Provincial Natural Science Foundation (Grant No.~ZR2020QA005).

\end{document}